\documentclass[12pt]{amsart}
\usepackage{geometry}                
\usepackage{graphicx}

\usepackage{amsmath}
\usepackage{amsfonts}
\usepackage{enumerate}
\usepackage{amssymb}
\usepackage{epstopdf}
\usepackage{tikz-cd}

\newtheorem{theorem}{Theorem}[section]
\newtheorem{proposition}[theorem]{Proposition}
\newtheorem{lemma}[theorem]{Lemma}
\newtheorem{corollary}[theorem]{Corollary}
\newtheorem{conjecture}[theorem]{Conjecture}

\theoremstyle{definition}
\newtheorem{definition}[theorem]{Definition}

\numberwithin{equation}{section}

\newcommand{\R}{\mathbb{R}}

\newcommand{\cF}{\mathcal{F}}

\renewcommand{\epsilon}{\varepsilon}

\renewcommand{\rho}{\varrho}

\DeclareMathOperator{\HC}{HC} 
\long\def\forget#1\forgotten{} 

\begin{document}

\title{Filling metric spaces}

\author[Liokumovich, Lishak, Nabutovsky, Rotman]{Yevgeny Liokumovich, Boris Lishak, Alexander Nabutovsky and Regina Rotman}
\date{}
\maketitle

\begin{abstract}
We prove a new version of isoperimetric inequality:
Given a positive real $m$, a Banach space $B$, 
a closed subset $Y$ of metric space $X$ and a continuous
map $f:Y \rightarrow B$ with $f(Y)$ compact 
$$\inf_F\HC_{m+1}(F(X))\leq c(m)\HC_m(f(Y))^{\frac{m+1}{m}},$$ where $\HC_m$ denotes the $m$-dimensional Hausdorff content, the infimum is taken
over the set of all continuous maps $F:X\longrightarrow B$ such that $F(y)=f(y)$ for all $y\in Y$,
and $c(m)$ depends only on $m$. Moreover, 
one can find $F$ with a nearly minimal $\HC_{m+1}$ such that its image 
lies in the $C(m)\HC_m(f(Y))^{1\over m}$-neighbourhood of $f(Y)$ with the exception of a subset with zero $(m+1)$-dimensional Hausdorff measure.

The paper also contains a very general coarea inequality for Hausdorff content and its modifications.

As an application we demonstrate an inequality conjectured by Larry Guth
that relates the $m$-dimensional
Hausdorff content of
a compact metric space with its $(m-1)$-dimensional Urysohn width. 
We show that this result implies new systolic
inequalities that both strengthen the classical
Gromov's systolic inequality 
for essential Riemannian manifolds
and extend this inequality 
to a wider class of non-simply connected manifolds.

\end{abstract}

\section{Introduction}

\subsection{ Isoperimetric inequalities.}

Given a metric space $X$ and a real $m>0$, one can define the {\it $m$-dimensional Hausdorff content} $\HC_m(U)$ of $U$, $U \subset X$, as the infimum of
$\sum_i r_i^m$ among all coverings of $U$ by countably many balls $B(x_i, r_i)$ in $X$. The definition of Hausdorff content 
looks similar to the definition of Hausdorff measure,
except that for Hausdorff content we do not take the limit
over all coverings with the maximal radius of the ball in the covering tending to $0$. In particular, the  $m$-dimensional Hausdorff
content of a set $U$ is always less than or equal
to the $m$-dimensional Hausdorff measure
of $U$. When the Hausdorff dimension of 
a compact metric space $U$
is greater than $m$ the $m$-dimensional Hausdorff measure is infinite, but $\HC_m(U)$
is always finite and can be very small. For integer $m$ the Lebesgue measure (volume) $vol(U)$ of a set with Hausdorff dimension $m$ satisfies
the inequality $\HC_m(U)\leq {1\over V_n(1)}vol(U),$ where $V_n(1)={\pi^{m\over 2}\over\Gamma({m\over 2}+1)}$ is the volume of a ball of radius $1$ in $\mathbb{R}^n$.


The main result of this paper is an isoperimetric inequality for Hausdorff contents.  
To state our main theorem
it will be convenient to introduce notation for the infimal Hausdorff content of an 
extension of a map between metric spaces. One can think of it as an analogue 
of the filling volume function. Given metric spaces $X,B$, subsets $Y \subset X$, $U \subset B$
and a map $f:Y \rightarrow U$ define
$\cF_m(f,X,U) = \inf \{\HC_m(F(X))\}$, where the infimum is taken over all maps
$F:X \rightarrow U$, such that $F|_Y = f$. If no such extension $F$ exists we set 
$\cF_m(f,X,U)=\infty$.

\begin{theorem} \label{isoperimetric0}
For each positive $m$ there exists a constant $const(m)$, such that
the following holds.
Let $B$ be a Banach space, $X$ metric space, $Y \subset X$ closed subset
and $f:Y \longrightarrow B$ a continuous map
with $f(Y)$ compact, then
$$\cF_{m+1}(f,X,B) \leq const(m)\HC_m(f(Y))^{m+1\over m},$$
where one can take $const(m)=(100m)^m$.
\end{theorem}

Note that $B$ may be infinite-dimensional, and we do not require $m$ to be an integer. Moreover, we will show (see Theorem \ref{isoperimetric}) that the isoperimetric
extension in Theorem \ref{isoperimetric0}
also obeys the following filling radius estimate:
Assume $\HC_{m}(f(Y)) \neq 0$ and let
$R= Const(m) \HC_m(f(Y))^{1\over m}$, where one can take $Const(m)=(1500m)^m$. 
In general, for any choice of $Const(m)$ it is not true that $f$ can be extended
to map $F$ whose image lies in the neighbourhood
$N_R(f(Y))$. Indeed, imagine the case when $Y$ is a very long, but very thin ``round" $2$-torus embedded in $\mathbb{R}^3$, and $X$ is the cone over this torus. The image of $X$ should at least fill the long parallel of the torus.

However, we prove that if $N_R(Y)$ is 
$(\lceil m \rceil -1)$-connected, then
$$\cF_{m+1}(f,X,N_R(f(Y))) \leq const(m)\HC_m(f(Y))^{m+1\over m}.$$
More generally, if $U$ is a $(\lceil m\rceil-1)$-connected open
set that contains $N_R(f(Y))$, then 
$$\cF_{m+1}(f,X,U) \leq const(m)\HC_m(f(Y))^{m+1\over m}.$$
Also, one can always find $F$ such that $\HC_{m+1}(F(X))\leq const(m)\HC_m(f(Y))^{m+1\over m},$ and $F(X)\setminus N_R(Y)$ is contained in a $\lceil m\rceil$-dimensional simplicial complex, where $const(m)$ and $Const(m)$ can be chosen as above
(see Theorem \ref{isoperimetric} below for the exact statement).

Finally, given a continuous map (e.g. an inclusion) $i:Y\longrightarrow B$ with compact image, $i$ can be ``filled" in $Const(m)\HC_m(i(Y))^{1\over m}$-neighborhood of $Y$. We define the {\it $m$-filling} for
 such $i$ and each positive $m$ as follows: An $m$-filling of $i$ consists of an open set $O\supset i(Y)$,
 a $(\lceil m\rceil-1)$-dimensional simplicial complex $K\subset O$, and a continuous map $\phi:Y\longrightarrow K$ 
 such that $i$ is homotopic to $\phi$ inside $O$. 
 
In section 3 we prove that there always exists
an $m$-filling of $i$ with $O$ in $Const(m)\HC_m^{1\over m}(Y)$-neighborhood of $Y$  such that $dist_B(i(y),\phi(y))\leq Const(m)\HC_m^{1\over m}(Y)$,
and $\HC_{m+1}(O)\leq const(m)\HC_m^{m+1\over m}(Y)$.
Now the general theory of extendability of continuous maps to topological spaces easily implies that
given a map of any $X\supset Y$ to $B$ that extends $i$, this map can be altered on $X\setminus Y$ to a map of $X$ into 
any contractible or even $(\lceil m\rceil-1)$-connected open set $O'$ that contains $O$. One can also get a map of $X$ extending $i$ into
the union of $O$ and the image of the cone over $K$ under a continuous map to $B.$


Now observe the following corollary of our results in the classical case of codimension one:

\begin{corollary} \label{codimone}
Let $M^n$ be a closed connected topological submanifold of $(n+1)$-dimensional Banach space $B^{n+1}$. Let $\Omega$ denote
the closure of the bounded connected component of $B^{n+1}\setminus M^n$. Then for each $m>0$, 
$\HC_{m+1}(\Omega)\leq const(m)\HC_m(M^n)^{m+1\over m}$.
\end{corollary}

\begin{proof}
Let $B=B^{n+1}$, $Y=M^n$, $X=\Omega$. Observe that the image of {\it any} filling of $M^n$ by $\Omega$ in $B^{n+1}$ must contain
the whole $\Omega$ (as $M^n$ is not null-homologous in $B^{n+1}$ minus any point in the interior of $\Omega$). Now the corollary
immediately follows from the previous theorem.
\end{proof}

\par\noindent
{\bf Remarks and questions.} Observe, that Corollary \ref{codimone} is an analog of the classical (codimension one) isoperimetric inequality.
The inequality in Corollary \ref{codimone} is interesting only for $m\leq n$. (If $m>n$, then both sides of the inequality are equal to $0$.)
So, assume that $m\leq n$. Now the first question is what is the optimal value of $const(m)$? We also can fix $B^{n+1}$ and $m$
and consider the optimal constant $c(B^{n+1},m)$ such that $\HC_{m+1}(\Omega)\leq c(B^{n+1},m)\HC_m(M^n)^{m+1\over m}$ for all open domains $\Omega \subset B^{n+1}$ bounded by a closed hypersurface $M^n$. It is interesting to find the exact values of $c(B^{n+1},m)$ for $B^{n+1}=\mathbb{R}^{n+1}$ or $l_\infty^{n+1}$. A naive conjecture is that $c(B^{n+1},m)$ is always equal to $1$. In fact, it is possible that the constant $const(m)$ in Theorem \ref{isoperimetric0}  and $Const(m)$ in its extension (Theorem \ref{isoperimetric})  are both equal to $1$.

The reasoning behind this conjecture is that one expect that the optimal constants in isoperimetric inequalities are achieved for metric balls. (One can explore this conjecture also for infinite-dimensional Banach spaces, especially, $l^\infty$.)
As a first step in this direction we will demonstrate that:

\begin{proposition} \label{isoperimetric constant}
$$c(l^{n+1}_\infty, n)=1.$$ In particular, if a closed topological manifold $M^n$ bounds a
domain $\Omega$ in $l^{n+1}_\infty$, then
$\HC_{n+1}(\Omega)\leq\HC_n(M^n)^{n+1\over n}$, and if $\Omega$ is a coordinate cube, then this inequality becomes an equality.
\end{proposition}

Similarly, one can investigate the dependence of the constant $C(B,m)=Const(m)$ in the upper bound $Const(m)\ \HC_m(Y)^{1\over m}$ for the radius $R$ of $N_R(f(Y))$. 
The case when this Banach space is isomorphic to $l^\infty$ is especially interesting. As the metric balls in $l^\infty$ are cubes, the most natural
conjecture is that $c(l^\infty, m)=C(l^\infty, m)=1$.
We can be even more specific.

\begin{conjecture}

Let  $Y$ be a compact subset of $l^\infty$. For each positive $m$ and each positive $\epsilon$  there exists an $m$-filling $(O,K,\phi)$ of (the inclusion of ) $Y$ in $B$
such that
\par\noindent
(a) $\HC_{m+1}(O)\leq \HC_m(Y)^{m+1\over m}+\epsilon$;
\par\noindent
(b) For each $y\in Y$ $\Vert y-\phi(y)\Vert_B\leq \HC_m^{1\over m}(Y)$.
\end{conjecture}

This conjecture might even be true for all ambient Banach spaces. The complex $K$ might arise as the nerve of a ``good" covering of $Y$ by metric balls
(or by smaller convex sets). To get a flavour of this conjecture, consider, for example,
a round  $2$-sphere of radius $1$ in $B=\mathbb{R}^3$. If one tries to cover it by very small balls,
as one does in order to determine the $2$-dimensional Hausdorff measure, the resulting upper bound for $\HC_2$ will be $4$ (the area of the sphere divided by the area of the  $2$-disc of radius 1.) However, the covering by one ball of radius $1$ yields
an upper bound equal to $1$. And, of course, the sphere contracts inside this ball
to a lower dimensional complex (e.g. the center of the ball).


If this conjecture is true, one can use the Kuratowski embedding to conclude that for each integer $m$
and each compact metric space $Y$ there exists a continuous map of $Y$ into 
an $(m-1)$-dimensional complex $K$ such that the inverse image of each point
can be covered by a metric ball of radius $\leq 2\HC_m(Y)^{1\over m}$ and that
$UW_{m-1}(Y)\leq 2\HC_m^{1\over m}(Y)$. This would  imply that a closed essential Riemannian manifold $M^n$ satisfies
the inequality 
$sys_1(M^n)\leq 4\HC_n^{1\over n}(Y)\leq {4\over V_n(1)^{1\over n}}vol^{1\over n}(M^n)= \sqrt{8\over \pi e}\sqrt{n}(1+o(1))vol^{1\over n}(M^n)$ 
(cf. \cite{GromovFilling} or \cite{N}).
Here $sys_1(M^n)$ denotes the minimal length of a non-contractible closed curve in $M^n$. We give Gromov's definition of essential manifolds in section 1.3 below. 
Not only the dependence on $n$ is optimal here, but looking at $\mathbb{R}\mathbb{P}^n$
with the canonical metric, we see that asymptotically the right hand side this inequality is within the
factor of $\pi$ from the optimal value. 
Thus, optimal (or even nearly optima) ``strengthened" isoperimetric inequalities for $l^\infty$ and $\HC_m$ would lead to a resolution of a long-standing conjecture in systolic geometry that asserts that the dimensional constant $c(n)$ in the right hand side of Gromov's systolic inequality behaves as $\sim\sqrt{n}$, when $n\longrightarrow\infty$.
The combinatorial nature of the problem and the extremely simple form of the conjectured optimal constants suggest that optimal isoperimetric inequalities for $\HC_m$ in $L^\infty$ might be easier to establish than their classical analogs for the Lebesgue measure.

\subsection{Inequalities relating Hausdorff content with Urysohn width.}

Let $X$ be a metric space, and $q$ a non-negative integer. One defines 
{\it $q$-dimensional Urysohn width} of $X$ as the infimum of $W$ such that there exist
 a $q$-dimensional simplicial complex $K$ and a continuous map $\pi: X \rightarrow K$ so that every fiber
$\pi^{-1}(s)$ has diameter $\le W$ in $X$. Intuitively, the Urysohn width $UW_{m-1}(X)$ measures
how well metric space $X$ can be approximated by 
an $(m-1)$-dimensional space.

As a part of our proof of Theorem \ref{isoperimetric0} we establish that:

\begin{theorem}\label{main2}
For each $m$ there exists a constant $c(m)$ such that each compact metric space $X$ satisfies the inequality
$$UW_{m-1}(X)\leq c(m)\HC_m^{1\over m}(X).$$
\end{theorem}

This result was first conjectured by L. Guth in \cite[Question 5.2]{Guth_Urysohn}.  In fact, Guth conjectured a stronger assertion that was
established in the first version of this paper (\cite{LLNR1}).
Recall, that metric space $X$ is called {\it boundedly compact}, if each closed metric ball in $X$ is compact
or, equivalently, each closed and bounded subset of $X$ is compact.



\begin{theorem} \label{main}
For each positive integer $m$ there exists $\epsilon_m>0$, such that the following holds.
If $X$ is a boundedly compact metric space and there exists a radius $R$, such that every ball of 
radius $R$ in $X$ has $m$-dimensional Hausdorff content less than $\epsilon_m R^m$,
then $UW_{m-1}(X) \leq R$. 
\end{theorem}

%
As observed by Guth in \cite{Guth_Urysohn}, Theorem \ref{main2} 
can be viewed as a quantitative version of the classical Szpilrajn theorem asserting that each compact metric space with zero $m$-dimensional Hausdorff measure has Lebesgue covering dimension $\leq m-1$. Indeed,
if the $m$-dimensional Hausdorff measure of $X$ is equal to zero, then, as $\HC_m$ does not exceed the $m$-dimensional Hausdorff measure, $\HC_m(X)$ also must be equal to zero. Now Theorem \ref{main2} implies that $UW_{m-1}(X)=0$,
which 
implies that the covering dimension of $X$ is at most $m-1$ (see the proof of Lemma 0.8 in \cite{Guth_Urysohn}).
Also, note that if $X$ is a closed $m$-dimensional Riemannian manifold, then Theorem \ref{main2} improves the
well-known Gromov's inequality relating the volume of a closed Riemannian manifold and its filling radius (as according to Gromov the filling radius 
does not exceed ${1\over 2}UW_{m-1}$ -- see \cite{GromovFilling}, pp. 128-129, where $UW_{m-1}$ is denoted as $Diam'_{m-1}$).

\begin{figure}
   \centering	
	\includegraphics[scale=0.7]{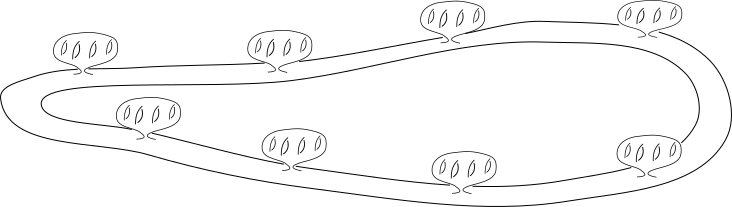}
	\caption{Surface $\Sigma$ is a connected sum of a thin long torus
	and many copies of a surface of very large area and very small diameter.
	The area of every metric ball of radius $1$ in $\Sigma$ is large,
	but 2-dimensional Hausdorff content is small. By Theorem \ref{main} the Urysohn width $UW_1(\Sigma)$ is also small.}
	\label{fig:example}
\end{figure}

Theorem \ref{main} generalizes 
a result of Guth in \cite{Guth_Urysohn}, where $X$ is assumed to be 
an $m$-dimensional Riemannian manifold, and $\HC_m$ replaced
by the volume.
This result has been previously conjectured by M. Gromov, and for a long time had been an open problem.
Guth's proof is based on a clever construction of a covering by balls with controlled overlap from
his previous paper \cite{Gu1} and also uses S. Wenger's simplified version (\cite{W})
of Gromov's proof of J.Michael-L.Simon isoperimetric inequality and its generalizations (\cite{GromovFilling}). 

If $X$ is compact, then choosing $R=({\HC_m(X)\over \epsilon_m'})^{1\over m}$ and denoting
${1\over (\epsilon_m')^{1\over m}}$ as $c(m)$ we see that Theorem \ref{main} implies Theorem \ref{main2}.

The proof of Theorem \ref{main} in \cite{LLNR1} is considerably longer than
the proofs of Theorems \ref{isoperimetric0} and \ref{main2} in the present paper.
Soon after \cite{LLNR1} was posted on arXiv, Panos Papasoglu learned about it from
the talk of one of the authors and our conversations with him during a conference at Barcelona. 
Later he came up with a much shorter proof of Theorem \ref{main} than ours (\cite{P}). While his
proof uses some technical and philosophical ideas from our work, its central idea
is different. He replaced Gromov's ``chopping off of thin fingers" argument (that we use in the present paper)
by an adapatation of Schoen-Yau style approach. (Note that Schoen-Yau approach was first introduced to systolic geometry by Guth; \cite{Gu2}). A later paper of one of the authors (see section 3 of \cite{N}) contains a simplification of Papasoglu's proof that not only proves Theorem \ref{main} on three pages but also yields an explicit linear upper bound for $\epsilon_m^{1\over m}$ (while 
our work or \cite{P} yielded only exponential upper bounds for $\epsilon_m^{1\over m}$).
This makes our original proof of Theorem \ref{main} obsolete, and we decided to
exclude from the present paper all sections that are not required for the proof of Theorem \ref{isoperimetric0} (the proof of
Theorem \ref{main2} is embedded in the proof
of Theorem \ref{isoperimetric0}).

On the other hand, the approach of \cite{P} does not seem applicable to prove our isoperimetric inequality, Theorem \ref{isoperimetric0}.
Moreover, one of the authors tried to combine the idea of \cite{P}
with the ideas of the present paper in order to either simplify
our proof or obtain a better value of the constant $const(m)$ in
Theorem \ref{isoperimetric0}, but was not able to do that.


\subsection{New systolic 
inequalities.}
We observed that Theorem \ref{main2} has the following corollary.

We first give the following definition, which
generalizes the notion of ``essential" polyhedra on p. 139 of \cite{GromovFilling}.

\begin{definition}
A CW-complex $X$
is $m$-essential 
if the classifying map 
$Q:X\longrightarrow K(\pi_1(X),1)$
is not homotopic to a map 
$Q'$ that factors through a map to an $(m-1)$-dimensional
$CW$-complex. 
\[
\begin{tikzcd}
X \ar[r,""] \ar[rr,bend right,"Q'"'] & P^{m-1} \ar[r,""] & K(\pi_1(X),1)
\end{tikzcd}
\]
\end{definition}

An $n$-dimensional manifold $M$ is called essential, as defined in \cite{GromovFilling}, 
 if there exists a coefficient group $G$, such that the homomorphism $Q_{*n}:H_n(M^n; G)\longrightarrow H_i(K(\pi_1(M^n),1);G)$
induced by the classifying map $Q:X\longrightarrow K(\pi_1(X),1)$ is non-trivial (that is, has a non-zero image).
According to definition above every 
essential manifold is $n$-essential.
Also, to prove that a manifold is $m$-essential
it is enough to find a coefficient group $G$ and $i \geq m$, such that one of the homomorphisms $Q_{*i}:H_i(X; G)\longrightarrow H_i(K(\pi_1(X),1);G)$ 
induced by the classifying map $Q:X\longrightarrow K(\pi_1(X),1)$ is non-trivial.

Recall that a metric is called a length metric, if the distance between each pair
of points is equal to the infimum of lengths of paths connecting these points, and a metric space such that its metric is a
length metric is called a length space.

\begin{theorem}
For each positive integer $m$ there exist a constant $C(m)$ with the following property.
Let $X$ be an $m$-essential
finite CW-complex endowed with a length metric.
Then
$$sys_1(X)\leq C(m)\HC_m^{1\over m}(X).\ \ \ \ \ \ (*)$$
\end{theorem}

\par\noindent
{\bf Remarks.} 
1. Compact $m$-essential Riemannian manifolds with or without boundary of dimension $n\geq m$ constitute the most obvious example of path metric spaces satisfying the conditions of the theorem. 
\par\noindent
2. If $m$ is the dimension of $X$, then this inequality
improves Theorem $B'_1$ on p. 139 of \cite{GromovFilling} (as $\HC_n(X)\leq V_n(1) vol(X)$).
When $X$ is a closed $m$-dimensional Riemannian manifold, this inequality is a strengthening of the famous Gromov systolic inequality $sys_1(M^n)\leq c(n)vol^{1\over n}(M^n).$
\par\noindent
3. Observe that if $k<m$ then $HC_k^{1\over k}(X)\geq HC_m^{\frac{1}{m}}(X),$
(see Lemma \ref{km} below).
Therefore, disregarding the constants $c(m)$, these inequalities become stronger as $m$ increases.
On the other hand, the assumption that $X$ is $m$-essential also becomes stronger when $m$ increases.
\par\noindent

We are going to prove this theorem in section 4. The proof combines the ideas from \cite{GromovFilling} with Theorem \ref{main2}.

{\bf Examples.}
1. If $E^m$ is an essential closed manifold (in the sense of \cite{GromovFilling}), then
any product $E^m\times N^{n-m}$ with a closed manifold, or a connected sum
$E^m\times N_1^{n-m}\# N_2^n$ is $m$-essential and satisfies the systolic inequality
from the previous theorem. For example, if $M^8$ is a Riemannian manifold 
diffeomorphic to $T^3\times S^5\# S^4\times S^4$, then $sys_1(M^8)\leq c(3)\HC_3(M^8)^{1\over 3}.$
\par\noindent
2. 
Observe that if $X$ is {\it not} $2$-essential, then the classifying map factors through some $1$-complex, and, therefore, $\pi_1(X)$ is free.
Therefore, if $X$ is a nonsimply-connected compact length space homeomorphic to a finite $CW$-complex, and $\pi_1(X)$ is {\it not} free, then
$sys_1(M^n)\leq C(2)\sqrt{\HC_2(M^n)}.$ 

The inequalities (*) can be restated as the assertion that there exists
a constant $b(m)>0$ such that if $\HC_m(X)\leq b(m)$ for an $m$-essential
$X$, then $sys_1(X)\leq 1$. Of course, Theorem \ref{main} implies more,
namely, 
that the assumption about the Hausdorff content of the whole manifold can be replaced by the assumption about all metric balls of radius $1$:
%
%
%
Also, as Theorem \ref{main} holds for boundedly compact metric spaces,
the previous theorem can be immediately generalized for locally finite CW-complexes.
In particular, it holds for $m$-essential complete non-compact Riemannian manifolds
with or without boundary.

\begin{theorem}
For each integer $m\geq 2$ there exists $b(m)>0$ with the following property:
Let $X$ be an $m$-essential
boundedly compact length space homeomorphic to a CW-complex. 
If for some $R>0$ for each metric ball $B$ of
radius $R$ in $X$ $\HC_m(B)\leq b(m)R^m$, then $sys_1(X)\leq R$.
\end{theorem}

\par\noindent
{\bf Example.} Consider a complete Riemannian metric on $T^2\times \mathbb{R}^N$. If $\HC_2$ of each metric
ball of radius $1$ does not exceed $b(2)$, then there exists a non-contractible closed
curve of length $\leq 1$.

\subsection{Hausdorff content: basic properties and some examples.}

In this section we list some properties of the
Hausdorff content for compact metric spaces that immediately follow from the definition:

\begin{enumerate}
    \item  Subadditivity. $\HC_m(\bigcup_i A_i)\leq\sum_i\HC_m(A_i).$

\item Good behaviour under Lipschitz maps.
If $f:X\longrightarrow Z$ is an $L$-Lipschitz map, and
$Y\subset X$, then $\HC_m(f(Y))\leq L^m\HC_m(Y).$ 

Applying this to the
inclusion map of $Y$ into $X$ (which is,
obviously, $1$-Lipschitz) we obtain:

\item Monotonicity. If $Y\subset X$,
then $\HC_m(Y)\leq \HC_m(X)$.

\item Rescaling. If $A$ is a subset of a Banach space, and $\lambda$ is a scalar, then $\HC_m(\lambda A)=\vert\lambda\vert^m\HC_m(A).$

\item For each $m$ we have $\HC_m(X)\leq rad^m(X)\leq  diam^m(X)$. 

Indeed, a compact metric space $X$ can always be covered by one metric ball of radius $diam(X)$.

\item Note the following basic inequality relating Hausdorff contents in different dimensions:
\begin{lemma} \label{km} If $m>k$ then $\HC_k^{1\over k}(X)\geq \HC_m^{1\over m}(X)$. 
\end{lemma}
Indeed, choose
a finite covering of $X$ by metric balls with radii $r_i$ so that $\sum_i r_i^k\leq\HC_k(X)+\epsilon$ for an arbitrarily small $\epsilon$. Now
$\sum_ir_i^m=\sum_i(r_i^k)^{m\over k}\leq (\sum_ir_i^k)^{m\over k}\leq (\HC_k(X)+\epsilon)^{m\over k}$. Now the desired inequality follows when we take $\epsilon\longrightarrow 0$.

\item Consider the Euclidean $2$-ball $B$ of
radius $1$. We already know that its $\HC_1$ cannot exceed $1$, but, in fact,
it is equal to $1$, as the sum of radii of any collection of balls covering a
diameter of $B$ (that has length $2$) should be at least $1$. 

\item $\HC_m$ is {\bf not additive}. Indeed,
cut the ball $B$ in the previous example into two halves $H_1$, $H_2$ along a diameter. Note, that $\HC_1(H_i)$
is still equal to $1$. So, when we 
remove one of these halves, the values
of $\HC_1$ does not decrease. Moreover, if we will take the union of the remaining half with a metric ball of arbitrarily small radius centered at one of two points of the diameter, the value of $\HC_1$ of the union will, in fact,
become greater than $1$. This example
illustrates difficulties that one can encounter trying to make the value of $\HC_m$ of a metric space smaller by cutting out its subsets and replacing them by subsets with a smaller value of $\HC_m$.

\item Let $B_K$ be a two-dimensional metric ball of radius $1$ in the hyperbolic space with constant negative curvature $K<< -1$.
The area of $B_K$ behaves as $\exp(\sqrt{-K})>>1$, yet $\HC_2(B_K)\leq 1<<vol(B_K),$ as $B_K$ can be covered
by just one metric ball of radius $1$.
It is obvious that $\HC_2(B_K)=1$, and that if we will cut a concentric metric
ball of a small radius $\epsilon$ out of $B_K$, the value of $\HC_2$ will not change. (This is another example of non-additivity of $\HC_m$, this time in
the situation when $m=dim(X)$.)

\item Let $l^n_\infty$
denote the $n$-dimensional linear space endowed with the $\max$-norm.
\begin{lemma} Consider the unit cube $C=[0,1]^n$ in $l^n_\infty$.
For each $m\leq n$ $\HC_m(C)={1\over 2^m}.$
\end{lemma}
Indeed, observe that the cube is the metric ball of radius ${1\over 2}$ centered at $({1\over 2},\ldots ,{1\over 2})$, and so $\HC_m(C)\leq ({1\over 2})^m$.

To prove the opposite inequality
consider a covering of $C$ by metric balls with radii $r_i$ so that
$\sum_i r_i^m\leq \HC_m(C)+\epsilon$
for an arbitrarily small $\epsilon$.
Observe that these metric balls are
cubes with side length $2r_i$. As the Lebesgue measure of their union cannot be less than the Lebesgue measure of $C$, we conclude that $\sum_i (2r_i)^n\geq 1$,
whence ${1\over 2^n}\leq \sum_i (r_i^m)^{n\over m}\leq (\sum_ir_i^m)^{n\over m}\leq (\HC_m(C)+\epsilon)^{n\over m}.$
When $\epsilon\longrightarrow 0$, we obtain the desired inequality.

\begin{corollary} \label{HCpar}
If $P=[0, r_1]\times\ldots\times [0,r_n]$, is an $n$-dimensional parallelepiped in $l^n_\infty$,
where $r_1\leq\ldots\leq r_n$, and $m\leq n$,
then $\HC_m(P)\geq ({r_1\over 2})^m.$
\end{corollary}

Indeed, by monotonicity, $\HC_m(P)$ is not less than $\HC_m$ of the $n$-cube $C \subset P$
with side length $r_1$, 
for which we have
$\HC_m(C)=({r_1\over 2})^m.$
\end{enumerate}

\subsection{Plan of the paper.}

Our proof of the isoperimetric inequality follows Gromov's approach (\cite{GromovFilling}) with simplifications of Wenger (\cite{W}) and modifications made
by Guth (\cite{Gu3}).
The additivity of volume (or, equivalently, Hausdorff measure)
is used many times throughout Gromov's, or Wenger's, or Guth's proofs.
Our biggest technical problem that the additivity
property of the Hausdorff measure does not hold anymore for
Hausdorff contents. Of course, the sole reason for failure of
additivity for $\HC_m$ is the fact that given
a cover of the union
of two disjoint sets $A$ and $B$ by metric balls, the same metric ball can cover both a part of $A$ and a part of $B$
providing ``savings" compared to the situation, when $A$ and $B$ are
covered separately.

Our first remedy involves the following observation: Assume that disjoint metric balls $B_1,\ldots ,B_k$
have comparable radii (say, all their radii are in the interval $[{s\over 2}, s]$ for some $s$), and $\HC_m$ of their union is
very small in comparison with $s^m$ (say, $< ({s\over 1000})^m$). This means that metric balls used for an almost optimal (from the point of view of $\HC_m$) covering of
$\bigcup_i B_i$ can contain only metric balls of radius $\leq {s\over 1000}.$ Now we can just throw away all metric balls
in this covering that intersect more than one of the balls $B_i$, and observe that the remaining balls still cover
$\bigcup_i(1-{1\over 250})B_i$, and no remaining ball can intersect more than one $B_i$. As a result, we can conclude
that $\HC_m(\bigcup_i B_i)$ is at least as large as $\sum_i\HC_m(0.996 B_i)$. If, in addition, we know that there is some $\theta>0$ such that $\HC_m(0.996 B_i)\geq\theta\HC_m(B_i),$
then we have the inequality $\HC_m(\bigcup_i B_i)\geq\theta\sum_i\HC_m(B_i)$ that provides a replacement for the missing additivity. A version of this argument is used in the proof of Proposition \ref{decomposition} below.

The proof of the main theorem continues only in 
section 3. Section 2 is devoted to the proof of the coarea inequality, the cone inequality
and a version of Federer-Fleming deformation theorem for
compact metric space, continuous functions and the Hausdorff content. These inequalities will
then be used in section 3 to prove all our main results.

The cone inequality involves two metric spaces $X$
and $Y\subset X$ inside of a ball of radius $R$ in a Banach space. One wants to construct a continuous map $\psi$ of $X$ into $B$
so that this map remains the identity map on $Y$, yet
$\HC_{m+1}(\psi(X))\leq c(m)R\HC_m(Y).$ The obvious idea is
to try to map $X$ into the cone $CY$ over $Y$ with the tip of the cone at the
center of $B(R)$. Yet how can one map $X$ into $CY$? The mapping to the cone would involve mapping at least some neighbourhood
of $Y$ in $X$ to $Y$, yet, there does not seem to be any general
way to do this. As a result, our ``cone" is {\it not} $CY$
(although it lies within a small neighbourhood of it).

Also, in section 2 we present a concise proof of a very general version of the coarea inequality
for Lipschitz maps in metric spaces and Hausdorff content. Moreover, it is applicable to coverings by arbitrary restricted families of balls, for example, balls with centers in a certain set, or balls of radius bounded above. We believe that our coarea inequality for Hausdorff content and its modifications has applications that go beyond the present paper.

Classical Federer-Fleming deformation theorem allows one
to construct a filling of a Lipschitz cycle with controlled mass
using repeated coning and projection
into skeleta of a fixed cubical grid. It also follows
from the construction that the filling lies in a neighbourhood
of the cycle of controlled size.
We prove a version of this result
for extensions of arbitrary continuous map from a compact metric space
and Hausdorff content instead of mass. A closely related results can be found in papers of Robert Young (\cite{Y}) and Guth (\cite{Gu3}).
The (not very significant) difference between our result and these
results is that these authors look for fillings of cycles by arbitrary chains, while we
are looking for a filling by a map from a fixed domain.

Section 3 contains our main results, and, in particular, an adaptation of the Michael-Simon-Gromov
isoperimetric inequality for Hausdorff contents. The original
J. Michael-L. Simon isoperimetric inequality \cite{MS} asserts that 
an $(n-1)$-dimensional cycle $Y$ in $\mathbb{R}^N$ bounds
an $n$-chain with $n$-dimensional volume $\leq c(n)vol_{n-1}(Y)^{n\over n-1}$. Gromov proved this
using a different method that makes it possible to prove this inequality
for $L^\infty$ and other Banach spaces instead of $\mathbb{R}^N$ (\cite{GromovFilling}).
Wenger simplified Gromov's proof and improved the value of the constant 
(\cite{W}). In \cite{Guth_Urysohn}
Guth adopted Wenger's proof to prove a version
of the Michael-Simon inequality for maps, where given a $n$-manifold $X$ with boundary $Y$ and its map $f$ into $\mathbb{R}^N$ one
wants to alter $f$ to a new map $F$ that coincides with $f$ on $Y$ and satisfies the inequality
$vol_n(F(X))\leq c(n)vol(f(Y))^{n\over n-1}$. Yet a similar theorem valid for maps to all Banach
spaces and with a concrete value for constant $c(n)$ is stated (without proof) already in \cite{GromovFilling} in section (A$'''$) of Appendix 2. In this theorem stated by Gromov both $X$ and $Y$ are assumed to be polyhedra rather than manifolds. In our Theorem \ref{isoperimetric} sets $X$ and $Y$ are assumed
to be compact metric spaces, and we use Hausdorff contents instead of Hausdorff measures.
Proofs of Gromov, Wenger
and Guth use induction on $n$. Our proof also uses
induction on $m$, despite the fact that unlike
the situation in all these papers, our $m$ has nothing to do
with the dimension of the compact metric space $Y$.

%
Our proof also provides an estimate of
how far the filling of $f(Y)$ by $F(X)$
will lie from $f(Y)$ (an analogue of Gromov's filling radius estimate).
Note, that it is {\it not} true that for
arbitrary $Y$ and $X$ the filling will be in $R$-neighbourhood of
$f(Y)$ for $R=const(m)\HC_{m}(Y)^{\frac{1}{m}}$.
Yet we prove that a filling by $X$ is possible in any $(\lceil m\rceil-1)$-connected open set containing the $R$-neighbourhood of $f(Y)$ and, moreover, $\HC_{m+1}(F(X))\leq const(m)\HC_m(f(Y))^{m+1\over m}$.


We want to follow the approach of \cite{W}, \cite{Guth_Urysohn} that involves a sequence of local improvements of the cycle $Y$ that we want to fill.
Assuming that $Y$ is a subset of $B$, and $f$ is the inclusion, we construct a sequence $Y=Y_1,Y_2,\ldots,$ of compact metric spaces
in $B$ and maps $\psi_k:Y_k\longrightarrow Y_{k+1}$. The distance between
$y\in Y_k$ and $\psi_k(y)$ is bounded by $const(m)\HC_m(Y_k)^{1\over m}$, and $\HC_m(Y_k)$ exponentially decreases with $k$. Therefore, for each $y\in Y$ the distance between $y$ and the image of $y$ under the composition of $\psi_i, i\in\{1,\ldots, k\}$ is uniformly bounded by $Const(m)\HC_m^{1\over m}(Y)$.

If the ambient
space $B$ has a finite dimension $N$, then once $\HC_m(Y_k)$ becomes less
than a certain $\epsilon$ that is allowed to depend also on $N$, one performs a version
of Federer-Fleming argument to find the last map $\psi_\infty: Y_k\longrightarrow Y_\infty$ into a $(\lceil m\rceil-1)$-dimensional CW-complex $Y_\infty$  so that each $y\in Y_k$
and its image under $\psi_\infty$ are close.

We explicitly construct extensions of the inclusion $f:Y \to B$ to parts of $X$ so that
the images of these maps will fill the ``gaps" between $Y_i$ and $Y_{i+1}$. The image
of one of these parts
of $X$ will partially fill between $Y_k$ and $Y_\infty$. The remaining part of $X$
will be mapped to the cone over $Y_\infty$. While we like the constructive
nature of this extension process, and it does not make our proof significantly more
complicated, we could have just forgotten about $X$, construct an $m$-filling
of $Y$ (defined in section 1.1) and use the theory of absolute extensors to conclude
that a desired extension to any $X$ always exists.

In this case, $Y_\infty$ would play the role of $K$ in the definition of $m$-filling, the composition of all $\psi_i$ the role of $\phi$, and an arbitrarily small open neighbourhood of an image of a homotopy between the inclusion $i$ of $Y$ and $\phi$
that naturally appears in our proof the role of $O$.

To ``improve" $Y$ (and then $Y_2$, $Y_3$, etc.) one finds a certain collection of
disjoint metric balls, cuts them out, and replaces them by chains with smaller volumes. On the first glance this approach seems absolutely bound to fail in our situation. Indeed, consider examples 8 and 9 in the previous subsection.
In both cases we cut out a large piece from some $Y$, yet $\HC$
does not decrease, and, moreover, can increase once one adds a set with an arbitrarily small Hausdorff content.

Yet a remedy exists: We combine the idea explained in (1) that involves using only the metric balls with $\HC_m$ that is much
less than the $m$th power of the radius with a powerful new idea: Fix an almost optimal (from the point of view of $\HC_m$)
finite covering $Q$ of $Y$ by metric balls $B_i$. We introduce the concept of $m$-dimensional Hausdorff content, $\widetilde\HC_m$, with respect to the covering $Q$ for subsets of $Y$. $\widetilde\HC_m$ is defined as $\HC_m$ but with the following important difference: This infimum of $\sum_i r^m$
is taken only over coverings by metric balls from $Q$.
(That is, each ball of any covering of a subset of $Y$ that we are allowed to consider must be one of the balls $B_i$.) 
Observe that: (a) $\HC_m\leq\widetilde\HC_m$, so an upper bound
for $\widetilde\HC_m$ is automatically an upper bound for 
$\HC_m$; (b) For $Y$ $\widetilde\HC_m$ and $\HC_m$ are almost
the same. Now the idea is to run Wenger's argument by
removing (and then replacing) only metric balls with
$\widetilde\HC_m\sim ({r\over A(m)})^m,$ where $r$ denotes the radius of the considered ball, and $A(m)$ is a sufficiently large constant.
(It is not difficult to see that if there are not sufficiently many metric balls with this property to run this argument,
then $Y$ has a ``round shape" (meaning that $\HC_{m}(Y)^{1\over m}\sim$
the diameter of $Y$ in the ambient metric), so that the isoperimetric inequality for $Y$ follows from our version of the cone inequality.) The basic idea
here is that now we are throwing out only ``important" balls.
Moreover, for each ball $B$ that is being replaced, the balls from $Q$ in the optimal covering of $B$ yielding the value of $\widetilde\HC_m(B)$ are also important (despite being very small): removal of each of those balls from $Y$ reduces $\widetilde\HC_m(Y)$ and, therefore, its $\HC_m$.

There is an additional difficulty in the proof that arises only in the case when the ambient Banach space is infinite-dimensional. The problem is that our version of Federer-Fleming deformation
theorem is proven only for finite-dimensional ambient Banach spaces.
Therefore, in order to prove the filling radius estimate we need to apply Federer-Fleming
deformation theorem
to a part of $Y$ of very small Hausdorff content by projecting it onto 
a finite dimensional subspace $S'$ of the Banach space. By Kadec-Snobar theorem
we can bound the Hausdorff content of the projection in terms of the dimension of
$S'$. For the argument to work we need to have an a priori bound for the dimension
of $S'$ and also know that the small part of $Y$ will lie close to $S'$.
To accomplish this we fix an almost optimal covering of $Y$ 
in the beginning of the argument and then choose
$S'$ that contains all centres of balls in the covering.
We then carry out the inductive construction
for a version of Hausdorff content measured
using ball with centres in $S'$.

Finally, we observe that our proof of Theorem \ref{isoperimetric0} yields not only Theorem \ref{isoperimetric0} but also Theorem \ref{main2}.

The last section contains the proof of the inequality $sys_1(M^n)\leq 3 UW_{m-1}(M^n)$ for $m$-essential manifolds that combined together with Theorem \ref{main2} yields Theorem 1.8. Our proof is 
similar to Gromov's proof of the inequality $sys_1(M^n)\leq 6 Fill Rad(M^n)$ in \cite{GromovFilling}.

\section{Cone, coarea and Federer-Fleming inequalities for Hausdorff content}

Recall the notation  for the infimal Hausdorff content of an 
extension of a map between metric spaces we introduced earlier.
Given metric spaces $X,Y$, subsets $A \subset X$, $U \subset Y$
and a map $f:A \rightarrow U$ define
$\cF_m(f,X,U) = \inf \{\HC_m(F(X))\}$, where the infimum is taken over all maps
$F:X \rightarrow U$, such that $F|_A = f$. If no such extension $F$ exists we set 
$\cF_m(f,X,U)=\infty$.

We will also need a version of $\cF$ where Hausdorff content is measured with 
respect to a collection of balls whose centers are contained in some fixed
subset $W$. Let $V, W \subset X$ be subsets of a metric space $X$.
Define $\HC_m(V;W)$ to be the infimum of
$\sum_i r_i^m$ among all coverings of $V$ by countably many balls $B(x_i, r_i)$ 
with $x_i \in W$. We set $\cF_m(f,X,U;W) = \inf \{\HC_m(F(X);W)\}$, where the infimum is taken over all maps
$F:X \rightarrow U$, such that $F|_A = f$, and $\cF_m(f,X,U;W)=\infty$ if
there is no extension.

\begin{lemma}[Cone inequality]\label{cone}
Let $B=B(p,R)$ be a closed metric ball of radius $R$ in a Banach space $S$, 
$W$ a linear subspace of $S$, $U$ a closed convex subset of $B(p,R)$ with $p \in W\cap U$.
Suppose $X$ is a metric space, $Y \subset X$ closed subset
and $\tau: Y \longrightarrow U$ is a continuous map with $\tau(Y)$ compact.
Then $$\cF_m(\tau, X,U) \leq m(1+{1\over m})^m R \HC_{m-1}(Y)\leq e m R\HC_{m-1}(Y).$$
More generally,
$$\cF_m(\tau, X,U;W) \leq m(1+{1\over m})^m R \HC_{m-1}(Y;W)\leq e m R\HC_{m-1}(Y;W)$$

Moreover, if $X \subset B$ and $\tau: Y \longrightarrow U$ is the inclusion map, $\rho>0$,
then for every $\delta>0$ there exists a map 
$F:X \longrightarrow U$ with
$$\HC_m(F(X);W) \leq m(1+{1\over m})^m R \HC_{m-1}(Y;W)+ \delta$$ and $F(X \setminus N_{\rho}(Y)) \subset \{p\}$.
\end{lemma}

Here $e = \lim_{m \rightarrow \infty}(1+ \frac{1}{m})^m=2.718\ldots$ is 
the Euler's number.

\begin{proof} 
We will prove the more general upper bound for $\cF_m(\tau, X,U;W)$.
By taking $W = S$ we get the bound for $\cF_m(\tau, X,U)$.

First we observe that the Lemma follows from its special case
when $X$ is a subset of $U$, and $\tau$ is the inclusion map (so $Y$ is a compact
subset of $U$).
Indeed, since $U$ is convex,
by a generalization of Tietze extension theorem proven by Dugundji
(see \cite[Theorem 4.1]{Du}) 
there exists an extension $F': X \longrightarrow U$
of $\tau:Y \longrightarrow U$.
If for every $\epsilon>0$ there exists a map $\Psi: F'(X) \longrightarrow U$
that is equal to the identity on the set $\tau(Y)$  and  
with $\HC_m(\Psi(F'(X))) \leq L + \epsilon$, then by considering composition $\Psi \circ F'$
we obtain $\cF_m(\tau, X,U) \leq L$.
Hence, without any loss of generality we may assume
$Y\subset X \subset U$ and that $Y$ is compact.

Let $\{B_i(r_i)\}_{i\in I}$ be a finite covering of
$Y$ by closed metric balls with centers in $W$,
such that $\sum_i r_i^{m-1}\leq \HC_{m-1}(Y;W)+\epsilon$ for some $\epsilon$ that
can be taken arbitrarily small. Without any loss of generality we can assume that the union of these balls
covers not only $Y$ but an open neighbourhood $N_r(Y)$ of $Y$ for a very small positive radius $r\leq\min_i r_i$. We take $r< \rho$ from the 
statement of the theorem and $r< \epsilon$.

For each set $V$ let $Cone_p(V)$ denote the union of the closed
straight line segments in the ambient Banach space connecting $p$ with
all points of $V$. 

Observe, that for each metric ball $B_i(r_i)$ the set $Cone_p(B_i(r_i))$
can be covered by at most $m{R\over r_i}$ metric balls of radius $(1+{1\over m})r_i$
with centers in $W$. 
Indeed, the first
of these balls is $B_i((1+{1\over m})r_i)$; the centers of all subsequent balls are spaced
along the segment connecting the center $q_i$ of $B_i(r_i)$ and $p$ at distances
${r_i\over m}$ apart. For every $t \in [0,1]$ the distance between $tq_i$
and the center point $y$ of one of these balls is at most ${r_i\over 2m}$. By triangle inequality,
if $x \in B_i(r_i)$ then $dist(tx,y) \leq t r_i + {r_i\over 2m}\leq (1+{1\over m})r_i$.

Therefore,
$Cone_p(N_r(Y))$ can be covered by a collection of at most $\sum_i m{R\over r_i}$ metric
balls with $m$ balls of radius $(1+{1\over m})r_i$ for each $i$, where $i$ ranges over $I$. Therefore,
$\HC_m(Cone_p(N_r(Y);W)\leq \sum_i {mR\over r_i}((1+{1\over m})r_i)^{m}\leq m(1+{1\over m})^m R(\HC_{m-1}(Y;W)+\epsilon),$
where $\epsilon$ is arbitrarily small.

Let $\phi: [0, \infty] \rightarrow [0,1]$ be a continuous monotone function with
$\phi(0) =1 $ and $\phi(t) = 0$ for all $t \geq r$ defined as
$1-{1\over r}x$ for $x\in [0,r]$ and zero for all $x\in [r,\infty)$.
We define
a map $\Phi: U \rightarrow Cone_p(N_{r}(Y)\cap U)$
by the formula $\Phi (x) = \phi(dist (x, Y)) x +(1-\phi(dist(x,Y))p$.

We set $Z = \Phi(X)$. As $Z\subset Cone_p(N_r(Y))$, we have $\HC_m(Z)\leq\HC_m(Cone_p(N_r(Y);W)\leq m(1+{1\over m})^m R(\HC_{m-1}(Y;W)+\epsilon)$
for arbitrarily small $\epsilon$. Also, it
is clear from the definition of $\Phi$
that $\Phi(x)=p$ if $dist(x, Y)\geq r \geq \rho$.
\end{proof}

{\bf Remark.} We can improve the inequality to 
$\cF_m(\tau, X,U;W) \leq const\ R\ \HC_{m-1}(Y;W)$
where $const$ is an absolute constant that does not depend on $m$ by using balls of radii $(1-{j\over m})(1+{1\over m})r_i$, $j=0,\ldots, m-1$ to cover $Cone_p(B_i(r_i))$
instead of the balls with identical radii $(1+{1\over m})r_i$. (The new balls are centered at the same points as the old balls.)
As the result, we improve multiplicative dimensional constants in the right hand sides of
all inequalities in the previous lemma from $m(1+{1\over m})^m$ to $2(1+{1\over m})^m< 2e<6$.


\vskip 0.1in

Observe that $Z$ in the proof of Lemma \ref{cone} is not a cone over $Y$, although it is very close to $Cone_p(Y)$.
In fact, we can replace the chosen value of $r$ by any smaller
positive value, and the proof would still work. So, if desired we could
choose $Z$ arbitrarily close to an actual cone over $Y$ in $S$.
Below we will be referring to $Z$  as the cone over $Y$, and will call this 
construction {\it the coning of Y}. 

We will also need the following co-area formula for Hausdorff content. Let $X$ be a metric space, and $\Gamma$ denote a set of closed metric balls
that together cover $X$. Three main examples of $\Gamma$ are: 1) All metric balls; 2) all metric balls with centers contained into some fixed subset $W$
of $X$; 3) A specific covering of $X$ by a family of closed metric balls; 4) All metric balls of radius $\leq R$, where $R$ is some constant. For $V\subset X$ we denote $\HC_m(V;\Gamma)$ to be the infimum of $\sum_i r_i^m$ among all coverings of $V$ by countably many balls $B(x_i,r_i)$, where $B(x_i,r_i)\in \Gamma$. In the case 1) $\HC_m(V;\Gamma)=\HC_m(V)$, in the case
2) $\HC_m(V;\Gamma)=\HC_m(V;W)$, in the case 3) when a covering of $X$ is fixed, we will be denoting $\HC_m(V;\Gamma)$ as $\widetilde{\HC}_m(V)$.

\begin{lemma} \label{coarea} (Co-area inequality)
Let $U$ be a compact set in an arbitrary metric space $X$, $\Gamma$ a family of closed metric balls that covers $X$, 
and $f: U \longrightarrow [R_1, R_2]$
a Lipschitz function with Lipschitz constant $Lip(f)$.
Then,
$$\int^{*R_2}_{R_1}\HC_{m-1}(f^{-1}(R);\Gamma)\ dR\leq 2Lip(f)\HC_m(U;\Gamma),$$
where $\int^*$ denotes the upper Lebesgue integral.
Therefore, there exists $R \in [R_1,R_2]$,
such that $\HC_{m-1}(f^{-1}(R);\Gamma) \leq \frac{2\ Lip(f)}{R_2 - R_1} \HC_m(U;\Gamma)$.
\end{lemma}

\begin{proof}
 Let $\{B_{r_i}(p_i)\}$,  $B_{r_i}(p_i)\in\Gamma$, 
 be a covering of $U$ with $\sum_i r_i^m\leq \HC_m(U;\Gamma)+\epsilon$, where $i\in\{1,\ldots ,N\}$ for some $N$, and $\epsilon$ can be chosen arbitrarily small. The desired inequality would follow from the inequality $\int^{*R_2}_{R_1} \HC_{m-1}(f^{-1}(R);\Gamma)\ dR\leq 2Lip(f)\sum_i r_i^m$.
We are going to prove a stronger inequality, where
$\HC_{m-1}(f^{-1}(R);\Gamma)$ is replaced by the following quantity that is obviously not less than $\HC_{m-1}(f^{-1}(R);\Gamma)$, namely, $\sum_{i\in I(R)}r_i^{m-1}$, where
$I(R)$ denotes the set of all indices $i$ such that the intersection of
$B_{r_i}(p_i)$ and $f^{-1}(R)$ is not empty. 
The left hand side of the desired inequality becomes
$$\int_{R_1}^{R_2}\sum_{i\in I(R)} r_i^{m-1}dR=\int_{R_1}^{R_2}\sum_{i=1}^N r_i^{m-1}\chi_i(R)dR=\sum_{i=1}^N r_i^{m-1}\int_{R_1}^{R_2}\chi_i(R)dR,$$
where the characteristic function $\chi_i(R)$ is equal to $1$ for
all $R\in [R_1,R_2]$ such that $f^{-1}(R)$ and $B_{r_i}(p_i)$ have a non-empty
intersection, and to $0$ otherwise. 

Finally, observe that 
if $\chi_i(\rho_1)=1$ and $\chi_i(\rho_2) =1$, then
$$|\rho_1-\rho_2| \leq Lip(f) dist(f^{-1}(\rho_1) \cap B_{r_i}(p_i), f^{-1}(\rho_2) \cap B_{r_i}(p_i))
\leq 2 Lip(f) r_i$$
It follows that $\int_{R_1}^{R_2}\chi_i(R)dR\leq 2Lip(f)r_i$, which implies the desired inequality.
\end{proof}

In particular, for a family of equidistant surfaces $\{\Sigma_R\}$,
$dist(\Sigma_t,\Sigma_s)=\vert t-s\vert$,
in a finite dimensional Banach space we obtain
the following version of co-area inequality.

\begin{lemma} \label{coarea2} (Co-area inequality 2)
Let $\{\Sigma_R\}$ be a family of equidistant surfaces in a Banach space $S$,
$W \subset S$ any subset, 
$U \subset S$ a compact set. Then,
$$\int^{*R_2}_{R_1}\HC_{m-1}(\Sigma_R\cap U;W)\ dR\leq 2\HC_m(U;W).$$
Therefore, there exists $R \in [R_1,R_2]$,
such that $\HC_{m-1}(\Sigma_R \cap U;W) \leq \frac{2}{R_2 - R_1} \HC_m(U;W)$.

\end{lemma}

\par\noindent
{\bf Example.} Consider the case, when $\Sigma_R$ are metric spheres $S_R$ (= sets of points at the distance $R$ from 
a point $p\in X$). Taking $W=X$, we see that 
$$\int^{*R_2}_{R_1}\HC_{m-1}(S_R)\ dR\leq 2\HC_m(A_{R_1,R_2}),$$
where $A_{R_1,R_2}$ denotes the set of points $x$ such that $dist(x,p)\in [R_1,R_2]$.

We will also need a version of the Federer-Fleming 
deformation theorem from \cite{FF} for Hausdorff content.
Instead of constructing an extension of $f$
into a cone over $Y$ this argument
gives an extension of $f$ into a portion of the cone
between $Y$ and a lower dimensional skeleton of a fixed cubical complex,
where the vertex of the cone is chosen using an averaging argument.
Pushing out in this way repeatedly one eventually constructs an isoperimetric
extension of $f$ whose image 
outside of a neighbourhood of $Y$ of controlled size is contained in the $(m-1)$-skeleton
of the cubical complex.
For controlling 
Hausdorff content of cycles
this construction appeared in 
\cite[Proposition 7.1]{Gu3} and \cite[Lemma 2.5]{Y}. Our proof follows the same strategy
as the proofs in \cite{Gu3} and \cite{Y}.
The difference is that instead of
constructing an $m$-chain filling an 
$(m-1)$-dimensional cycle
we construct an extension of an arbitrary continuous map $f$ from a subset $Y$ of a metric space $X$. 

Let $Q(r)$ denote the cubical complex in $\R^n$ whose 
$n$-dimensional faces  are 
$\{C(a_1, ... , a_n) = \{[a_1r, (a_1+1)r] \times ... \times [a_nr, (a_n+1)r]; a_i \in \mathbb{Z}\}$, and let $Q(r)^{(k)}$ denote the $k$-skeleton of $Q(r)$,
$0 \leq k \leq n$. 

Let $\lceil m \rceil$ denote the smallest
integer that is greater than or equal to 
$m$.

\begin{lemma}[Federer-Fleming type inequality] \label{Federer-Fleming}
Let $X$ be a metric space, $Y\subset X$
and $f: Y \longrightarrow \R^n$ continuous map with $f(Y)$ compact.
Then for every $m \in [2,n]$ there exists a constant $c_1(n)>0$, such that
$$\cF_m(f,X, \R^n) \leq c_1(n) \HC_{m-1}(f(Y))^{\frac{m}{m-1}}$$
Moreover, there exists a constant $c_2(n)$, such that for every $\delta>0$
and $R = c_2(n) \HC_{m-1}(f(Y))^{\frac{1}{m-1}}+ \delta$ we have
$$\cF_m(f,X, N_R(f(Y)) \cup Q(R)^{(\lceil m \rceil-1)}) \leq c_1(n) \HC_{m-1}(f(Y))^{\frac{m}{m-1}}$$
\end{lemma}

\begin{proof}
As in the proof of Lemma \ref{cone} we may assume, without any loss of generality,
that $Y\subset X \subset \R^n$, $Y$ is compact and $f$ is the inclusion map. Fix small $\delta>0$. 
Let $Q(R)$ denote a cubical complex as defined above. We will show that for some constant $c_2(n)$ and
$R= c_2(n)\HC_{m-1}(f(Y))^{\frac{1}{m-1}}+ \delta$
the statement of the Lemma holds.

We can define an open set $Y'$ with smooth boundary, a very small $\tau>0$
and $Z = \overline{N_{\tau}(Y') \setminus Y'} \cong \partial Y' \times [0,1]$, such that $Y \subset Y'$
and $\HC_{m-1}(N_{\tau}(Y')) < \HC_{m-1}(Y)+ \delta$.
Indeed, this follows by slightly thickening
an almost optimal covering of $Y$ by balls.

We will define a map $\Psi: \overline{N_{\tau}(Y')} \rightarrow \R^n$, such that the following holds:
\begin{itemize}
    \item[(a)] $\Psi(x) = x$ for $x \in Y'$;
    \item[(b)] $\Psi(\partial N_{\tau}(Y')) \subset Q(R)^{(\lceil m \rceil -2)} $;
    \item[(c)] $dist(x, \Psi(x)) \leq const(n) R$;
    \item[(d)] $\HC_m(\Psi(\overline{N_{\tau}(Y')} )) \leq const(n) R (
    \HC_{m-1}(Y)+ \delta)
    \leq c_1(n) \HC_{m-1}(f(Y))^{\frac{m}{m-1}} + \delta$.
\end{itemize}
Since $\Psi \big( \partial (\R^n \setminus N_{\tau}(Y')\big) \subset Q(R)^{(\lceil m \rceil-2)}$,
we can extend $\Psi$ to a map from $\R^n$ to
$Q(R)^{(\lceil m \rceil-1)} \cup \Psi(N_{\tau}(Y'))$.
Restricting $\Psi$ to $X \subset \R^n$
gives us the desired extension of $f$.

To define map $\Psi$ on the set $Z \cong \partial Y' \times [0,1]$ we will define a sequence of homotopies of $\partial Y'$
into successively smaller dimensional skeleta $Q(R)^{(k)}$, $k = n, ... , \lceil m \rceil-2$.

Let $C$ be a $k$-face in $Q(R)^{(k)}$, $p \in C$ and 
let $P_{C,p}: C \setminus \{p\} \rightarrow \partial C$
denote the radial projection map. 
Suppose $k \geq \lceil m \rceil -1$ and $V \subset C$.
By the ``pushout from an average point'' argument 
\cite[Lemma 7.2]{Gu3} there exists a positive constant $c_0(k)$, such that if 
\begin{equation} \label{pushout}
\HC_{m-1}(V) \leq c_0(k) R^{m-1}
\end{equation}
then there exists
a  point $p \in C \setminus V$, such that $\HC_{m-1}(P_{C,p}(V)) \leq
const(k) \HC_{m-1}(V)$.  
Define a homotopy $\Phi_{p,V}: [0,1] \times V \rightarrow C$ given by
$$\Phi_{p,V}(t, x) = (1-t) x + t P_{C,p}(x)$$
By cone inequality 
we then have
$$\HC_m(\Phi_{p,V}([0,1]\times V))\leq const(k)\  R\  \HC_{m-1}(V) $$
Note that the homotopy $\Phi_{p,V}$ is constant on $V \cap \partial C$.

We can define homotopies $\Phi_k:[0,1] \times V_k \rightarrow Q(R)^{(k)}$, $k=n,...,\lceil m \rceil-1$
as follows. 
We set $V_n= \partial Y'$ and $V_k = \Phi_{k+1}(1, V_{k+1}) \subset Q(R)^{(k)}$ for $k<n$.
On each $k$-face $C$ of $Q(R)^{(k)}$ define $\Phi_k= \Phi_{p(V_k \cap C),V_k \cap C}$,
where the point $p(V_k \cap C) \in C$ is chosen as described above. 
We claim that $V_k \cap C$ satisfies (\ref{pushout}), so $\Phi_k$
is well-defined. 
Note that $\Phi_k$ is continuous on all of $V_k$ 
since the homotopy is constant on the boundary of $C$.

To prove the claim we observe the following. Let $\{ \beta(p_j,r_j) \}$ denote a nearly optimal
covering of $V_k$, that is $\sum_j r_j^{m-1} \leq
\HC_{m-1}(V_k)+ \delta$. If $\HC_{m-1}(V_k) \leq R^{m-1}$,
then each ball $\beta(p_j,r_j)$ intersects at most
$const(n)$ faces of $Q(R)^{(k)}$. Hence, if 
$\{C_m\}$ are the $k$-faces of $Q(R)^{(k)}$, then
$$\sum_m \HC_{m-1}(V_k \cap C_m) \leq const(n) \HC_{m-1}(V_k)$$
Choosing $c_2(n)$ in the definition of $R$ sufficiently large
we can ensure that inequality (\ref{pushout})
is satisfied if we replace $V$ with $V_k \cap C$ for each $k=n,...,\lceil m \rceil-1$
and each face $C$ of $Q(R)^{(k)}$.
With this choice of $c_2(n)$ 
we have that homotopies $\Phi_k$ are well-defined for $k=n,...,\lceil m \rceil-1$.

Performing homotopies $\Phi_k$ we obtain a homotopy of $\partial Y'$
into $Q(R)^{(\lceil m \rceil-2)}$. 
This defines map $\Psi$ on $Z$ and define $\Psi$ to be the identity on $Y'$.
It follows from the construction that condition
$(a) - (d)$ are satisfied.
\end{proof}

Observe that in the course of the proof of the previous proposition we
also established (cf. \cite[Proposition 7.1]{Gu3}) that:

\begin{lemma} \label{Federer-Fleming radius}
Let $Y\subset \mathbb{R}^n$ be compact, $m\in [1,n-1]$. There exists a compact $(\lceil m\rceil-1)$-simplicial
complex $Q\subset \mathbb{R}^n$ and a map $\rho:Y\longrightarrow Q$
such that for each $y\in Y$ $dist(y,\rho(y))\leq const(n)\HC_m(Y)^{1\over m}$.
\end{lemma}

Here $m$ corresponds to $m-1$ form the previous lemma; $Q$ is a bounded
subcomplex of $Q(R)^{(\lceil m\rceil-1)}$. 

Also, for every $n$-dimensional normed
linear space $L$  there exists a linear isomorphism $T:L\longrightarrow\mathbb{R}^n$
such that $\Vert T\Vert \Vert T^{-1}\Vert\leq \sqrt{n}$ (F. John's theorem).
Therefore, two previous lemmae
can be imeediately generalized to an arbitrary $n$-dimensional ambient normed linear space
instead of $\mathbb{R}^n$.

\section{Isoperimetric extension inequality} \label{Iso}
In this section we prove the isoperimetric inequality 
Theorem \ref{isoperimetric0}.
Recall that given $Y \subset X$ and a map $f: Y \rightarrow U$
we denote by $\cF_m(f,X,U)$ the infimum of $\HC_m(F(X))$ over all extensions
$F:X \rightarrow U$ of the map $f$ and by $\cF_m(f,X,U;W)$ the analogous infimum of $\HC_m(F(X);W)$,
where Hausdorff content is measured with 
respect to balls whose center lies in $W$.
The following result can be thought of 
as a version of an isoperimetric 
and filling radius inequalities in Banach spaces (compare with Theorem A$'''$ in Appendix 2 of \cite{GromovFilling}).

Recall, that a CW complex $P$ is called $k$-connected
if $P$ is connected and $\pi_j(P) = 0$ for
$j=1,...,k$.

\begin{theorem}\label{isoperimetric}
Let $m\in(1, \infty)$, $S$ be a Banach space,
$U\subset S$ a closed ball, 
$X$ metric space and
$Y \subset X$ closed subspace.
Suppose $f: Y \longrightarrow U$ is a continuous map
with $f(Y)$ compact.
There exists constant $I_1(m)>0$, such that 
$$\cF_m(f,X, U) \leq  I_1(m)\HC_{m-1}(f(Y))^{m\over m-1}.$$

Moreover, we have the following estimate for
the filling radius. There exists a constant
$I_2(m)>0$ with the following property.
For every $\delta>0$ let $R= I_2(m)\HC_{m-1}(f(Y))^{1\over m-1}+\delta$
and let $\epsilon>0$ be arbitrary. 
There exists an extension $F: X \longrightarrow U$
of $f$, such that 
$$\HC_m(F(X)) \leq I_1(m)\HC_{m-1}(f(Y))^{m\over m-1} + \epsilon,$$
and an $(\lceil m \rceil -1)$-dimensional CW complex $W \subset U$,
such that $F(X) \setminus N_R(f(Y)) \subset W$.
If $N_R(f(Y)) \subset P \subset U$, where $P$ is
an $(\lceil m \rceil-2)$-connected open set
, then we can ensure that $W\subset P$, and, therefore,
$$\cF_m(f,X,P) \leq  I_1(m)\HC_{m-1}(f(Y))^{m\over m-1}.$$

\end{theorem}

\medskip
Our proof yields explicit values for constants $I_1(m)$ and $I_2(m)$. One can take 
$I_1(m)=(100m)^m$ 
and $I_2(m) = (1500m)^{m}$. 
Here $S$ may be infinite dimensional
and $m$ non-integer. Observe, that constant $\epsilon$ and $\delta$ are required in this theorem only for the case, when
$\HC_{m-1}(Y)=0$. If $\HC_{m-1}(Y)\not= 0$, then $\epsilon$ and $\delta$ can be omitted from the formulae in the text of the theorem.
\medskip

%

\textit{Proof of Theorem \ref{isoperimetric}}.
First observe that the last sentence of
Theorem \ref{isoperimetric} follows
from the second last sentence. Indeed,
if $W$ is $(\lceil m \rceil -1)$-dimensional CW complex in $U$
and $P$ is $(\lceil m \rceil -2)$-connected, then there exists a
Lipschitz map $F':W \longrightarrow P$
that is equal to the identity on
$P \cap W$. The statement follows by
composing $F'$ with $F$.

For an arbitrarily small $\epsilon>0$ 
consider a covering of $Y$ by finitely many balls $\{ \beta(q_i, \rho_i) \}$,
such that $\sum \rho_i^{m-1} \leq \HC_{m-1}(Y) + \epsilon$.
Let $S' \subset S$ be a finite-dimensional Banach space that contains $\bigcup q_i$
and the center of ball $U$.

We fix $S'$ for the rest of the proof.
The reason why we need $S'$ is the following. 
Assume we proved the result for $\HC_l$ with $l \leq m$ and we want to prove it for $l=m+1$. Similarly to the argument of Wenger,
using repeated coning
in a certain carefully chosen collection of balls $\{ B_j \}$
and inductive assumption applied to
$\partial B_j \cap Y$
we will reduce the extension 
problem to the case when $Y$ has arbitrarily 
small Hausdorff content. In a similar situation Wenger \cite{W} also applies
coning to finish the proof. In our case coning will not give us the desired bound for the filling radius.
Instead, we use projection onto $S'$. 
By the Kadec-Snobar theorem this projection will increase the Hausdorff content 
at most by a factor of $const(dimension(S'))$.
We then apply Federer-Fleming extension Lemma \ref{Federer-Fleming}
(see subsection \ref{final step}). In order for this argument to work 
we need to fix finite dimensional subspace $S'$ in the beginning
of the proof.

Observe that to prove Theorem \ref{isoperimetric}
it is now enough to construct an extension $F$ of the map $f$, such that 
$F(X) \subset (N_R(Y) \cup W) \cap U$, where $W$ is some $(\lceil m \rceil-1)$-dimensional
CW complex, for 
$R \leq I_2(m) \HC_{m-1}(Y;S')^{\frac{1}{m-1}} + \delta$ and 
$\HC_{m}(F(X)) \leq I_1(m)\HC_{m-1}(f(Y);S')^{m\over m-1}+ \epsilon$.
In particular, we can state the following 
Proposition that implies Theorem \ref{isoperimetric}.

\begin{proposition} \label{reduction isop}
Let $m\in(1, \infty)$, $S$ be a Banach space, $U \subset S$ a closed ball, 
$S' \subset S$ a finite-dimensional linear
subspace, 
$X$ metric space, $Y \subset X$ closed subset and $f: Y \longrightarrow U$ continuous map with $f(Y)$ compact.

There exist constants $I_1(m)>0$ and $I_2(m)>0$,
such that for all $\delta>0$ and 
$\epsilon>0$ the following holds.
Let $R= I_2(m)\HC_{m-1}(f(Y);S')^{1\over m-1}+\delta$.
There exists an extension $F: X \longrightarrow U$
of $f$, such that 
$$\HC_m(F(X);S') \leq I_1(m)\HC_{m-1}(f(Y);S')^{m\over m-1} + \epsilon,$$
and an $(\lceil m \rceil -1)$-dimensional CW complex $W \subset U$,
such that $F(X) \setminus N_R(f(Y)) \subset W$.
\end{proposition}

The rest of the proof will be devoted to 
demonstrating Proposition \ref{reduction isop}.
We will proceed by induction on $m$. 

\subsection{The base case $m \in (1,2]$.} \label{base case}

We can apply a well-known generalization of Tietze extension theorem
to locally convex target spaces
(see \cite[Theorem 4.1]{Du}) to find a map $\Phi: X \longrightarrow U$ that coincides with
$f$ on $Y$.
As in the proof of Theorem \ref{cone} we observe that 
Proposition \ref{reduction isop}
immediately follows from its particular case, when $X\subset U$, as we can just apply this particular case to
the sets $Y'=f(Y)$ and $X'=\Phi(X)$ and the inclusion map $f':f(Y) \hookrightarrow U$
and then compose the resulting map that has small Hausdorff content of the image with $\Phi$. 
So without any loss of generality we assume in this subsection that
$Y \subset X\subset U$ and $f$ is the inclusion.

We start from a general overview of our strategy. We are going to partition $X$ into several pieces
and define maps from each piece to $U$ that can be combined into a single
continuous map $F$ from $X$ to $U$ 
and will have the desired properties. 

All but one of these ``pieces" will be small closed neighbourhoods of 
some disjoint closed metric balls in $S$ with centers in $S'$
providing an almost optimal covering of $Y$ (from the
perspective of $\HC_{m-1}$ with centers in $S'$). These pieces will be mapped using the coning construction, so that their boundaries will be mapped to points $p_i$. Finally, it would remain to map the last ``piece", i.e. the closure of the complement to
the (finite) union of all previously described ``pieces". The map on its boundary
is already defined; it sends the boundary to a finite collection of points, and one can  extend it to a map into an arc connecting all these points using Tietze extension theorem.

Let $\{ B_i \}_{i=1}^N$ be a finite collection of closed metric balls covering $Y$ with 
centers in $S'$ and
 $\sum r_i^{m-1} \leq \HC_{m-1}(Y;S')+ \delta$. Without any loss of generality we may assume that each $B_i$ contains a point 
 of $Y$.
 
Observe, that, as $m-1\leq 1$, $\sum r_i^{m-1}\geq (\sum_i r_i)^{m-1}$. Using this observation,
it is straightforward to prove by induction 
on the number of balls in $N$ in the covering that there exists another covering $\{\beta_j\}_{j=1}^{N'}$, $N'\leq N$ with centers in $S'$
such that balls $\beta_j$ are pairwise disjoint and, still, $\sum R_i^{m-1} \leq \HC_{m-1}(Y;S')+ \delta$, where $R_i$ denotes the radius of $\beta_i$.
Indeed, if two balls $B(p_k,r_k)$,  $B(p_l,r_l)$ of the covering
intersect, consider the ball with the radius $r_k+r_l$ centered at the point $p$
on the straight line segment $[p_kp_l]$ at the distance $r_k$ from $p_2$. It is obvious, that $p\in S'$ and $B(p,r_k+r_l)$ contains
both balls $B(p_k,r_k)$,  $B(p_l,r_l)$ . On the other hand $(r_k+r_l)^{m-1}\leq r_k^{m-1}+r_l^{m-1},$ so we can replace both this balls
by $B(p,r_k+r_l)$. Now we can repeat this procedure until all balls of the covering will become disjoint.

Now denote centers and radii of $\beta_j$ as $p_j$ and $R_j$. Choose a small positive $\rho$ such that the slightly larger concentric closed balls
$B(p_j,R_j+\rho)$ still do not pairwise intersect.

Consider a map from $\partial B(p_j,R_j+\rho)\cup (Y\cap \beta_j)$ to $\beta_j$ that sends $\partial B(p_j,R_j+\rho)$ to $p_j$ and is the identity map on $(Y\cap \beta_j)$. (Observe that our assumption about $\rho$ implies that  $Y\cap \beta_j=Y\cap B(p_j, R_j+\rho)$.)
Cone inequality Lemma \ref{cone} implies that this map can be continuously extended to the map $B(p_j,R_j+\rho)\longrightarrow \beta_j$. Denote the restriction
of this map to $X\cap B(p_j,R_j+\rho)$ by $F_j$. By construction $F_j$ sends $X\cap\partial B(p_j,R_j+\rho)$ to the center $p_j$ of $\beta_j$.
Note that $\HC_{m-1}$ of the image of $F_j$ cannot exceed $\HC_{m-1}(\beta_j)\leq R_j^{m-1}$.
Further, $\HC_{m-1}(\cup_j F_j(X\cap B(R_j+\rho))\leq \sum_j R_j^{m-1}\leq \HC_{m-1}(Y)+\delta$.
Also, each point of $\cup_j F_j(X\cap B(R_j+\rho))$ is $(R_j+\rho)$-close to one of the points $p_j$. As $R_j\leq \sum_j R_j\leq (\sum_j R_j^{m-1})^{1\over m-1}\leq (\HC_{m-1}(Y)+\delta)^{1\over m-1}$, one can choose a sufficiently small $\delta$ such that $\cup_{j=1}^{N'} F_j(X\cap B(R_j+\rho))$ is $(\HC_{m-1}^{1\over m-1}(Y)+\epsilon)$-close to the $0$-dimensional complex $\{p_1,\ldots , p_{N'}\}$.

Connect $p_1 \in S'$ to points $p_2$, $p_3$, etc, by straight line segments so that we obtain a finite contractible 1-complex $W$.
The map $\cup_j X\cap\partial B(p_j,R_j+\rho)\longrightarrow \cup_j p_j$ can be continuously 
extended to a map $F_{N'+1}:X\setminus\cup_{j=1}^{N'} int(B(p_j,R_j+\rho)\longrightarrow W$. (Here $int(B)$ denotes
the interior of $B$; recall that all our balls are closed.) 
Indeed, one can choose a small  $\mu$ such that all closed balls $B(p_j, R_j+\rho+\mu)$ are disjoint. We can map the part of $X$ outside of the union of these balls to $p_1$. Each point $x$ of $X$
in the annulus $B(p_j,R_j+\rho+\mu)\setminus B(p_j+R_j+\mu)$ will be mapped
to $p_j+{dist(x, p_j)-R_j-\rho\over \mu}(p_1-p_j)$.)
Obviously, $d$-dimensional Hausdorff content
of the image of $F_{N'+1}$ is zero for any $d>1$, and, in particular, for $d=m$.

By construction, we can combine $F_j$, $j=1,\ldots, N'+1$ into one continuous map $F$ of $X$ into $S$. For $m \in (1,2]$ we have $$\HC_m(F(X))=\HC_m(F_j(\cup_{j=1}^{N'}B(p_j,R_j+\rho))),$$
since $F_{N'+1}(X\setminus\cup_{j=1}^{N'} int(B(p_j,R_j+\rho))) \subset W$
has $0$ $m$-dimensional Hausdorff content.
In particular, $\HC_m(F(X))\leq  (\HC_{m-1}(Y)+\delta)^{m\over m-1}$.

\subsection{Inductive step}
By induction we assume the conclusions of 
Proposition \ref{reduction isop}
to be true for all dimensions less than 
or equal to $m$. We will now prove it for $m+1$. 
The following Proposition is the analogue of
the Proposition from \cite{W}.

\begin{proposition} \label{decomposition}
Let $A(m)>m$ and
$Y' \subset S$ be a compact subset. 
For every $\epsilon>0$ there exists a finite set of disjoint balls $\{B_j = B(p_j,r_j) \}$ with centers in $S'$
and a constant $\alpha \in (\frac{1}{12},1]$,
such that the following inequalities hold:

\begin{equation}\label{eq:main0}
    \max_j r_j \leq (1+ \frac{1}{m})^2 A(m) \HC_m(Y';S')^{1\over m} + \epsilon
\end{equation}

\begin{equation}\label{eq:main1}
 \HC_m(Y' \setminus \bigcup B_j;S')\leq (1-\alpha^m)\HC_m(Y';S') + \epsilon
\end{equation}

\begin{equation} \label{eq:main2}
    \sum_j r_j \HC_{m-1} (\partial B_j \cap Y';S')^{\frac{m}{m-1}}
    \leq {200\ 4^{1\over m-1}m\alpha^{m+1} \over A(m)^{{1\over m-1}}} \HC_m(Y';S')^{\frac{m+1}{m}}
    + \epsilon 
\end{equation}

\begin{equation} \label{eq:main3}
    \sum_j \HC_{m-1} (\partial  B_j \cap Y';S')^{\frac{m}{m-1}}
    \leq {50m\ 4^{1\over m-1}\alpha^{m} \over A(m)^{{m\over m-1}}}  \HC_m(Y';S')+ \epsilon 
\end{equation}

\begin{equation} \label{eq:main4}
    \sum_j r_j \HC_m (B_j \cap Y';S')
    \leq 20\alpha^{m+1} A(m)  \HC_m(Y';S')^{\frac{m+1}{m}}+\epsilon .
\end{equation}

\end{proposition}

\begin{proof}
The main difficulty in the construction of the covering is 
that Hausdorff content is not additive for disjoint sets. 
This can be circumvented by considering a modification of Hausdorff content
of subsets of $Y'$ with
respect to a fixed nearly optimal covering $Q$ of $Y'$.
We will find a collection of disjoint balls $\{B_j = B(p_j, r_j)\}$, such that each $B_j$ contains a slightly
smaller ball $B(p_j, r(p_j))$  with the property
that optimal coverings of $B(p_j, r(p_j))$ and $B(p_i, r(p_i))$, $i \neq j$, by balls chosen from $Q$, are disjoint. This will imply approximate additivity of the modified Hausdorff content for
$\{B_j\}$.

Fix a finite covering of $Y'$ by closed balls $\beta_k$, $k=1,\ldots ,N$ (for some $N$) of radius $\tilde{r}_k$
and with centers in $S'$, so that
$\sum_{k=1}^N \tilde{r}_k^m \leq \HC_m(Y';S')+\epsilon$, where $\epsilon$ can be chosen arbitrarily small. Without any loss of generality
we can assume that no ball $\beta_k$ (even with a very small radius) is contained in the union
of the other balls of the collection. (This requirement can, obviously, be satisfied by inductively removing such ``unnecessary" balls from the collection.) We denote this collection of balls
$Q$, and the center of $\beta_k$ by $q_k$. For each subset $W$ of $Y'$ define $m$-dimensional Hausdorff content $\widetilde{\HC}_m(W)$ with respect to $Q$ as the infimum of $\sum_{k\in J} \tilde{r}_k^m$ over all subsets $J\subset\{1,\ldots ,N\}$ such that $W\subset \bigcup_{k\in J}\beta_k$. In other words, we calculate the Hausdorff content with respect to only balls from the collection $Q$. 
Clearly, for each $W$ we have $\widetilde{\HC}_m(W)\geq \HC_m(W;S')$, so any upper bound
for $\widetilde{\HC}_m$ will be automatically an upper bound for $\HC_m$.
On the other hand, it immediately follows from definitions that $\widetilde \HC_m(Y')\leq \HC_m(Y',S')+\epsilon.$

Fix a point $p \in Y'$ and consider quantity $\lambda_p(r) = {\widetilde{\HC}(B(p,r)\cap Y')\over r^m}$.
(Recall that $B(p,r)$ denotes a closed ball of
radius $r$ centred at $p$).
Observe that the following properties
of $\lambda_p(r)$ directly follow from the definition:
\begin{enumerate}
    \item $\lambda_p(r)$ is piecewise continuous;
    \item $\lim_{x \rightarrow r^-}\lambda_p(x) =\lambda_p(r) \leq \lim_{x \rightarrow r^+}\lambda_p(x)$;
    \item $\lim_{x \rightarrow 0}\lambda_p(x) = \infty$;
    \item $\lim_{x \rightarrow \infty}\lambda_p(x) = 0$;
    \item For each $k\in\{1,\ldots ,N\}$ we have $\lambda_{q_k}(\tilde r_k)=1$.
\end{enumerate}
The third equality holds since $\widetilde{\HC}$ of any non-empty subset of $Y'$ is bounded from below by the m-th power of the radius
of the smallest ball.
For each $p$ we define
$r(p)=\sup \{r\vert \lambda_p(r) \geq {1 \over A(m)^m}\}$. 
Observe that there exists a sequence
of radii $r_l$ approaching $r(p)$ from below with 
$ \lambda_p(r_l) \geq {1\over A(m)^m}$;
on the other hand, 
$ \lambda_p(r) <{1\over A(m)^m}$ for every 
$r>r(p)$. It follows from the properties
of $\lambda_p(r)$ that

\begin{equation} \label{eq:lambda(r(p))}
    \lambda_p(r(p)) = {1\over  A(m)^m}.
\end{equation}

Property (5) of $\lambda$ implies that for each center $q_k$
we have $r(q_k)>\tilde r(q_k)$.

Also, for every $\theta \geq 1$ we have

\begin{equation} \label{eq:lambda2}
\lambda_p(\theta r(p)) < \lambda_p(r(p)) = {1\over  A(m)^m}.
\end{equation}

Now we would like to define concentric balls $B(p, \overline{r}(p))$
for somewhat larger
radii $\overline{r}(p) \geq r(p)$. For this purpose 
consider the annulus
$A=B(p, (1+{1\over m})^2r(p))\setminus B(p, (1+{1\over m})r(p))$. Applying the coarea inequality (Lemma \ref{coarea})
to $A\cap Y'$ we see that there exists $\overline{r}(p)\in [(1+{1\over m})r(p), (1+{1\over m})^2r(p)]$
such that 

\begin{equation} \label{eq:coarea}
\widetilde\HC_{m-1}(\partial B(p, \overline{r}(p)) \cap Y')\leq {2 m^2\over (m+1)r(p)}\widetilde{\HC}_m(B(p, (1+{1\over m})^2r(p)) \cap Y'),
\end{equation}

\noindent
(observe, that our coarea inequality applies to $\widetilde \HC$).

Observe that, as $\overline{r}(q_k)>r(q_k)>\tilde{r}(q_k)$,
the collection of balls $B(q_k, \overline{r}(q_k))$, $k=1, \ldots, N$,
covers $Y'$.
Therefore, one can use the Vitali covering construction
to find a finite set of disjoint balls 
$B_j=B(q_{i_j}, \overline{r}(q_{i_j}))$, $j\in\{1,\ldots, L\}$ from the collection $Q$, so that
$\bigcup_{j=1}^L B(q_{i_j}, 3 \overline{r}(q_{i_j}))$ covers all $Y'$. 
We set $p_j = q_{i_j}$ and
$r_j = \overline{r}(q_{i_j})$. Define {\it the density constant} $\alpha$ as 
$$\alpha = \big(\frac{\sum_{j=1}^L\widetilde{\HC}_m(B(p_j, r(p_j))\cap Y')}{\widetilde{\HC}_m(Y')}\big)^{1\over m}$$ 
We claim that $\alpha\in (\frac{1}{12},1]$, and
the set of balls $\{B_j\}_{j=1}^L$, ($L\leq N$),
satisfies inequalities (\ref{eq:main0})-(\ref{eq:main4}).

It will be convenient to define $\eta_j(\theta) =  \widetilde{\HC}_{m}(B(p_j,\theta r(p_j)) \cap Y')$
and $\theta_j = \frac{r_j}{r(p_j)} \in [1+\frac{1}{m},(1+\frac{1}{m})^2]$.
Since $\eta_j(\theta) = \theta^m r(p_j)^m \lambda_{p_j}(\theta)$ we can write inequality 
(\ref{eq:lambda2}) as follows.
For $\theta \geq 1$

\begin{equation} \label{eq:comparison with larger ball}
\eta_j(\theta) \leq \theta^m \eta_j(1) =
{\theta^m r(p_j)^m \over  A(m)^m}.
\end{equation}

Note that $\alpha= \big({\sum_{j=1}^L \eta_j(1)\over \widetilde{\HC}_m(Y')}\big)^{\frac{1}{m}}$.

\textit{Proof of inequality (\ref{eq:main0}).}
Using the definition of $r_j = \overline{r}(p_j)$ and
(\ref{eq:comparison with larger ball}) we have
$$r_j = \theta_j r(p_j) \leq (1+ \frac{1}{m})^2 A(m) 
\eta_j(1)^{1\over m}
\leq 
(1+ \frac{1}{m})^2 A(m) \HC_m(Y')^{1\over m} + O(\epsilon).$$

\textit{Proof of inequality (\ref{eq:main1}) and the inequalities $\frac{1}{12} < \alpha\leq 1$.}
The key observation to prove the
rest of the inequalities is that
we have additivity of $\widetilde{\HC}_m$
for disjoint sets $B(p_j, r(p_j)) \cap Y' \subset B_j \cap Y'$.
Indeed, let $\{\beta_{k_l} \}$ denote the covering
of $B(p_j, r(p_j)) \cap Y'$ 
realizing its Hausdorff content with respect to $Q$,
$\widetilde \HC_m(B(p_j, r(p_j)) \cap Y') = \sum_l rad(\beta_{k_l})^m$.
By (\ref{eq:comparison with larger ball}) we have
$$rad(\beta_{k_l})\leq \eta_j(1)^{\frac{1}{m}} = \frac{r(p_j)}{A(m)} < \frac{r(p_j)}{m}$$
for each $l$. In particular, if $i\not= j$, then none of
balls $\beta_{k_l}$ can appear as a ball $\beta_{m_{l'}}$ 
in a covering that realizes $\widetilde{\HC}_m(B(p_i, r(p_i))\cap Y')$.
(Indeed, recall that the concentric balls $B(p_i, \overline{r}(p_i))$ and
$B(p_j,\overline{r}(p_j))$ are disjoint, and $\overline{r}(p_i)\geq r(p_i)+\frac{r(p_i)}{m}$, $\overline{r}(p_j)\geq r(p_j)+\frac{r(p_j)}{m}$.)
On the other hand, every ball in $Q$ is necessary for $Q$ to be a covering of $Y'$. 
Therefore, the $m$th powers of the radii of all balls $\beta_{k_l}$
appear in the expression for $\widetilde HC_m(Y')$.
Hence, $\sum_j \eta_j(1)
\leq \widetilde \HC_m(Y'),$ or, equivalently, $\alpha\leq 1$.

Also, $\beta_{k_l}$ does not intersect 
$Y' \setminus B(p_j,r_j)$.
It follows that 
\begin{equation} \label{eq:difference}
\widetilde \HC_m(Y' \setminus \bigcup  B_j) 
    \leq \widetilde \HC_m(Y') - 
\sum_j \eta_j(1)=(1-\alpha^m)\widetilde \HC_m(Y'),
\end{equation}
which immediately implies (\ref{eq:main1}).
On the other hand, since $Y' \subset \bigcup B(q_i, 3\overline{r}(q_i))$ 
and utilizing (\ref{eq:comparison with larger ball}) as well as the inequalities $\overline{r}(p_j)\leq r(p_j)(1+\frac{1}{m})^2$ for all $j$, we have
$$\alpha^m \widetilde \HC_m(Y')
= \sum_j\eta_j(1)\geq \sum_j \frac{\eta_j(3(1+{1\over m})^2)}{3^m(1+{1\over m})^{2m}}
\geq {1\over 3^m (1+{1\over m})^{2m}} \widetilde \HC_m(Y'). $$
Hence, $\alpha \geq {1\over 3(1 + \frac{1}{m})^2} > {1 \over 12}$ (as $m> 1$).
We conclude that
\begin{equation} \label{eq:sum of eta}
\frac{1}{12} < \alpha \leq 1
\end{equation}

\textit{Proof of inequality (\ref{eq:main2}).}
From inequalities (\ref{eq:coarea}) and (\ref{eq:comparison with larger ball}) we obtain 

\begin{equation} \label{eq:estimate2a}
\begin{split}
   \widetilde\HC_{m-1}(\partial B_j \cap Y')
& \leq {2m^2\over (m+1) r(p_j)}\eta_j((1 + \frac{1}{m})^2) \\
& \leq {2m^2\over (m+1) r(p_j)}\eta_j\big((1 + \frac{1}{m})^2\big)^{\frac{1}{m}} 
\eta_j\big((1 + \frac{1}{m})^2\big)^{\frac{m-1}{m}} \\
& \leq \frac{2(m+1)} {A(m)} \eta_j\big((1 + \frac{1}{m})^2\big)^{\frac{m-1}{m}} \\
& \leq \frac{2(m+1)(1+\frac{1}{m})^{2m-2}}{A(m)} \eta_j(1)^{\frac{m-1}{m}}\\
& \leq {2e^2m \over A(m)} \eta_j(1)^{\frac{m-1}{m}}.
\end{split}
\end{equation}

We use (\ref{eq:estimate2a}), (\ref{eq:comparison with larger ball}) and the definition of $\alpha$
to bound from above the
the left hand side of 
(\ref{eq:main2}).

\begin{equation} \label{eq:estimate2b}
\begin{split}
\sum_j r_j \widetilde\HC_{m-1} (\partial B_j \cap Y')^{\frac{m}{m-1}}
& \leq \sum_j  \big( {2(m+1)(1+\frac{1}{m})^{2m-2} \over A(m)} \big)^{\frac{m}{m-1}}\theta_j r(p_j) \eta_j(1)\\
& \leq  \big( { 2(m+1)(1+\frac{1}{m})^{2m-2}\over A(m)} \big)^{\frac{m}{m-1}}(1 + \frac{1}{m})^2 A(m) \sum_j  \eta_j(1)^{\frac{m+1}{m}} \\
& \leq  {200\ 4^{1\over m-1}m  \over A(m)^{{1\over m-1}}}  \big(\sum_j \eta_j(1)\big)^{\frac{m+1}{m}} \\
& \leq {200\ 4^{1\over m-1}m\alpha^{m+1} \over A(m)^{{1\over m-1}}} \widetilde \HC_m(Y')^{{m+1\over m}} \\
& \leq {200\ 4^{1\over m-1}m\alpha^{m+1} \over A(m)^{{1\over m-1}}} \HC_m(Y';S')^{{m+1\over m}} +O(\epsilon).
\end{split}
\end{equation}
Estimate $(2(m+1)(1+\frac{1}{m})^{2m-2})^{\frac{m}{m-1}}(1+\frac{1}{m})^2\leq 200m\ 4^{1\over m-1}$
in the third line
can be verified by an elementary but tedious calculation using the inequality $m>1$. 
This finishes the proof of (\ref{eq:main2}).

\noindent
\textit{Proof of inequalities (\ref{eq:main3}) and
(\ref{eq:main4}).}
Using (\ref{eq:estimate2a})  and the definition of $\alpha$ we obtain:

\begin{equation*} 
\begin{split}
      \sum_j \widetilde\HC_{m-1} (\partial \tilde B_j \cap Y')^{\frac{m}{m-1}} 
    &  \leq \big({2(m+1)(1+\frac{1}{m})^{2m-2}\over A(m)}\big)^{{m\over m-1}} \sum_j \eta_j(1) \\
    & \leq {50m\ 4^{1\over m-1}\alpha^m \over A(m)^{{m\over m-1}}}  \HC_m(Y';S')+O(\epsilon) .
\end{split}
\end{equation*}

Using (\ref{eq:comparison with larger ball}), the definition of $\alpha$, and the inequality $\theta_j\leq (1+{1\over m})^2$, we get

\begin{equation*}
    \begin{split}
\sum_j r_j\widetilde \HC_{m} (B_j \cap Y')
& = \sum_j \theta_j r(p_j)
\eta(\theta_j)\\
& \leq \sum_j r(p_j)\theta_j^{m+1} \eta_j(1)\\
& =\sum_j \theta_j^{m+1}A(m) \eta_j(1)^{m+1 \over m}\\
& \leq 20A(m)\alpha^{m+1}\HC_m(Y';S')^\frac{m+1}{m} + O(\epsilon).
    \end{split}
\end{equation*}
\end{proof}

Let $I_1(m)=(100m)^m, A(m)=[100m4^{1\over m-1}I_1(m)]^{m-1\over m}<(100m)^m$,
$I_2(m)=10m12^mA(m)<(1500m)^m$.


To finish the proof of Proposition \ref{reduction isop} we will need
the following definition:

Given an open set $H \subset X$ and
an extension $F: X \longrightarrow U$
of $f$, we will say that 
$(H,F)$ is an $(\alpha,\epsilon)$-improvement pair
for $X,Y,f$ if the following conditions hold
($\bar H$ below denotes the closure of $H$):

\begin{enumerate}
    \item $\HC_{m+1}(F(X\cap \bar H);S')\leq \frac{1}{4}I_1(m)\alpha^{m+1}
    \HC_m(f(Y);S')^\frac{m+1}{m}+ \epsilon$;
    \item $\HC_{m}(F(X\cap \partial H) \cup F(Y\setminus H);S') \leq (1-\frac{\alpha^m}{2})
    \HC_{m}(f(Y);S')+ \epsilon$;
    \item $F(\bar H) \subset N_R(Y)$ for 
    $R \leq 3 A(m) \HC_{m}(f(Y);S')^{\frac{1}{m}} + \epsilon$.
\end{enumerate}

\begin{lemma} \label{improvement pair}
Assume the same set up as in Proposition
\ref{reduction isop}. For every $\epsilon>0$ there exists $\alpha \in [\frac{1}{12},1]$ and
an $(\alpha,\epsilon)$-improvement pair $(H,F)$ for
$X,Y,f$.
\end{lemma}

\begin{proof}
As in the subsection \ref{base case} we can assume, without any loss of generality,
that $Y \subset X \subset S$ and $Y$ is compact.
Apply Proposition \ref{decomposition}
to $Y'=Y$ to obtain a set of disjoint closed balls $\{B_j \}$ with centers
in $S'$, and then determine the density constant $\alpha$. We take $H=\cup_j int(B_j)$, where $int(B_j)$ denotes the open ball with the same center and radius as $B_j$.

Now, first, we define $F$ on $\bigcup_j X \cap \partial B_j$.
This is accomplished by applying (using our inductive assumption) 
Proposition \ref{reduction isop}
one dimension lower with $X$,
$Y_j = Y\cap  \partial B_j$ playing the role of $Y$,
and $B_j$ playing the role of ball $U$.
We obtain a map 
$\tau_j: X\rightarrow B_j$ satisfying

\begin{equation} \label{eq:tau}
    \HC_m (\tau_j(X)
    ;S') \leq I_1(m) \HC_{m-1}( \partial B_j \cap Y;S')^{\frac{m}{m-1}} .
\end{equation}

\begin{figure}
   \centering	
	\includegraphics[scale=1]{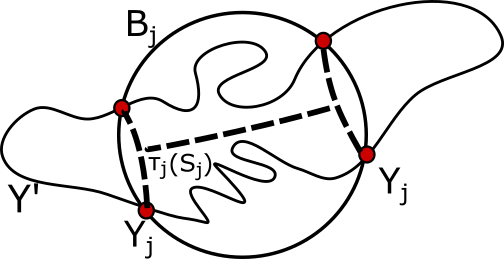}
	\caption{}
	\label{fig:isoperimetric}
\end{figure}

Now we would like to restrict $\tau_j$ to $S_j=X\cap \partial B_j$,
and use these mappings to define $F$ on $\bigcup_j X \cap B_j$.
Let $Z_j = (Y \cap B_j) \cup \tau_j(S_j)$.
We apply Lemma \ref{cone} to define
a cone over $Z_j$ map
$F_j:  X \cap B_j \longrightarrow B_j$, such that 
$F_j= \tau_j$ on $S_j$
and $F_j$ is the identity on $Y \cap B_j$.
This map satisfies the inequality $\HC_{m+1}(F_j(X \cap  B_j);S') \leq e m r_j \HC_m (Z_j;S')$.

Using (\ref{eq:tau}) and (\ref{eq:main2}), (\ref{eq:main4})
from Proposition \ref{decomposition} we estimate
\begin{equation} \label{eq:coning estimate}
\begin{split}
\sum_j \HC_{m+1}(F_j(X \cap { B_j});S') 
& \leq \sum_j e m r_j \big( \HC_{m} (\tau_j(S_j);S')+
\HC_m(Y \cap {B}_j;S') \big) + \epsilon \\
& \leq \sum_j e m r_j \big(I_1(m) \HC_{m-1}(\partial  B_j \cap Y;S')^{\frac{m}{m-1}} + \HC_m(Y \cap {B}_j;S') \big)+ \epsilon\\
& \leq \big({200 e \ 4^{1\over m-1}m^2  I_1(m)\over A(m)^{{1\over m-1}}}+20 e m A(m)\big)\alpha^{m+1} \HC_m(Y;S')^{{m+1\over m}} + \epsilon \\
& < (100m)^{m+1}
\alpha^{m+1}\HC_m(Y;S')^{{m+1\over m}} + \epsilon\\
&=(\frac{m}{m+1})^{m+1}I_1(m+1)\alpha^{m+1} \HC_m(Y;S')^\frac{m+1}{m}+\epsilon .\\
\end{split}
\end{equation}


We estimate the Hausdorff content of
$Y_1= (\bigcup_j \tau_j(S_j)) \cup (Y \setminus \bigcup_j B_j)$.
Combining inequalities (\ref{eq:tau}) and (\ref{eq:main1}), (\ref{eq:main3})
from Proposition \ref{decomposition}
we get
\begin{equation} \label{eq:HCY}
\begin{split}
    \HC_m(Y \setminus \bigcup_j  B_j;S')
+ \sum_j \HC_m(\tau_j(S_j);S') &\leq \big((1-\alpha^m) + {50 m 4^{1\over m-1}\alpha^m I_1(m)\over A(m)^{{m\over m-1}}}\big)\HC_m(Y;S')+ \epsilon \\
& \leq (1-\frac{\alpha^m}{2})\HC_m(Y;S')+ \epsilon\\
& \leq (1-\frac{1}{2*12^m})\HC_m(Y;S')+ \epsilon .
\end{split}
\end{equation}
by our choice of $A(m)$. (Recall that $\alpha\geq {1\over 12}$.)
We also have that 
by (\ref{eq:main0}) for all $j$ the image of
$F_j(X_j)$ is contained in $N_R(P)$ for 
$R \leq \max_j r_j \leq 3 A(m) \HC_{m}(Y;S')^{1\over m} + \epsilon$.
\end{proof}
\par\noindent
{\bf Remark.} To prove Theorem \ref{main2} we will need
the following consequence of the above construction.
Given a compact $Y\subset U\subset S$, let $\tilde Y=(Y\setminus\bigcup_j B_j)\cup (\cup_j\tau_j(Y\cap B_j))$, where
$B_j$ and $\tau_j$ were defined in the proof of the previous lemma.
Now we can define
a continuous map $\theta:Y\longrightarrow \tilde Y$ defined
as (the restriction of) $\tau_j$ on $Y\cap B_j$ for each $j$, and the identity map
outside of $\bigcup_j B_j$. For each point $y\in Y$ $dist (y,\theta(y))$
does not exceed $\max_j r_j\leq 3A(m)HC_m(Y;S')^{1\over m}+\epsilon.$
Also, using the fact that $\tau_j(Y\cap B_j)$ is contained in $\tau_j(X)$, inequality (\ref{eq:tau}), 
and, finally, the same argument as in inequality (\ref{eq:HCY}),  we see that $\HC_m(\tilde Y;S')\leq (1-{1\over 2*12^m})\HC_m(Y;S')+\epsilon$.
We can record our observation as the following lemma:
\begin{lemma}\label{stepsY}
Given a Banach space $S$, its finite-dimensional subset $S'$, a closed ball $U$ in $S$, a compact subset $Y$ of $U$, and an arbitrarily small positive $\epsilon$, there exists a compact $\tilde Y\subset S$ and a continuous map $\theta:Y\longrightarrow \tilde{Y}$ such that
\begin{enumerate}
    \item $\HC_m(\tilde Y;S')\leq (1-{1\over 2*12^m})\HC_m(Y;S')+\epsilon$; and
    \item For each $y\in Y$ $\Vert y-\theta(y)\Vert\leq 3A(m)\HC_m(Y;S')^{1\over m}+\epsilon,$ where $A(m)<(100m)^m$.
\end{enumerate}
\end{lemma}
\vskip 0.5truecm
We continue our proof of Proposition \ref{reduction isop}.
We now define the desired extension
$F$ by repeatedly applying 
Lemma \ref{improvement pair}.

Fix $\epsilon>0$.
We set $X_1 = X$, $f_1=f$ and $Y_1 =Y$.
Inductively, let $(H_k,F_k)$
be an $(\alpha_k,\epsilon_k)$-improvement pair defined in Lemma \ref{improvement pair} for $X_k,Y_k,f_k$, where
$\alpha_k \in [\frac{1}{12},1]$
and $\epsilon_k$ is a sequence of positive numbers rapidly converging to $0$, so that $\sum_{k=1}^\infty \epsilon_k$ can be made arbitrarily small in comparison with
$\epsilon$. More specifically, we can choose $\epsilon_k=\frac{\epsilon}{3m10^m A(m)2^k}$. (Note, that the notation $F_k$ has a new meaning now that
differs from its meaning in the proof of Lemma \ref{improvement pair}.)

We let $X_{k+1}= X_k\setminus H_k$,
$Y_{k+1} = \partial H_k \cup (Y_k \setminus H_k)$
and $f_k$ be the restriction of 
$F_k$ to $Y_k$.
For a large  $K>1$ and for 
each $k\leq K$ we define map $F: \bigcup_{k=1}^{K} H_k \longrightarrow S$ by setting $F(x) = F_k(x)$.
Note that this map is well-defined and
continuous on $\bigcup_{k=1}^K H_k$. Indeed, if
$x \in H_{k_1} \cap H_{k_2}$ for $k_1<k_2$,
then $x \in \partial H_k$ for all $k_1 \leq k \leq k_2$
and $F_{k_2}(x) = F_{k_1}(x)$. We also have that
$F(x) = f(x)$ for all $x \in Y \cap \bigcup H_k$.

By property 1) of $(\alpha, \varepsilon_k)$-improvement pairs we
have that for an arbitrarily large 
$K$, 
the following inequality holds:

\begin{equation} \label{eq:induction estimate}
  \begin{split}
     \HC_{m+1}\big(F(\bigcup_{k=1}^K H_k);S'\big) 
     &\leq \sum_{k=1}^K \HC_{m+1} \big(F_k( H_k);S'\big )\\
     &\leq (\frac{m}{m+1})^{m+1}I_1(m+1) \alpha^{m+1}
     \sum_{k=1}^K (\HC_m(Y_k;S')^\frac{m}{m+1} +\epsilon_k)\\
     &\leq (\frac{m}{m+1})^{m+1} I_1(m+1) \alpha^{m+1}
     \big(\sum_{k=1}^\infty (1-\frac{\alpha^m}{2})^{\frac{(k-1)m}{m+1}}\big)
     \HC_m(Y;S')^\frac{m}{m+1}+\sum_k\epsilon_k  \\
     &  =(\frac{m}{m+1})^{m+1}\frac{\alpha^{m+1}}{1-(1-\frac{\alpha^m}{2})^{\frac{m}{m+1}}} I_1(m+1) \HC_m(Y;S')^\frac{m+1}{m}+ \sum_k\epsilon_k  \\
     & < I_1(m+1)\HC_m(Y;S')+\frac{\epsilon}{2}.
  \end{split}
\end{equation}

(The last inequality can be easily derived once one notices that the generalized
binomial theorem implies that the denominator
$1-(1-\frac{\alpha^m}{2})^{\frac{m}{m+1}}> \frac{m}{m+1}\frac{\alpha^m}{2}$.)

To estimate the filling
radius observe that $F(\bigcup_{k=1}^K H_k) \subset
N_R(Y)$, for $R$ satisfying

\begin{equation}\label{eq:cumulative distance}
\begin{split}
R & \leq \sum_{k=1}^K R_{k} \leq 3A(m)\sum_{k=1}^K (\HC_m(Y_k;S')^{\frac{1}{m}}+\epsilon_k)\\
& \leq 3A(m) \sum_{k=1}^\infty (1-\frac{1}{2*12^m})^{k-1 \over m}
\HC_m(Y;S')^{\frac{1}{m}}+\epsilon\\
& \leq 10m12^mA(m)\HC(Y;S')^{\frac{1}{m}}+\epsilon \leq
I_2(m)\HC_m(Y;S')+\epsilon .
\end{split}
\end{equation}
Last inequality follows by our
choice of $I_2(m)$.

Now we are going to digress and observe that we have obtained 
the following lemma: 

\begin{lemma} \label{moving Y}
Let $Y$ be a compact subset of a closed ball in Banach space
$S$, $S'$ a finite-dimensional subspace of $S$. There exists a sequence of compact subsets $Y^{(i)},$ $i=1,2,\ldots $ with $Y^{(1)}=Y$ and maps $\theta_i:Y^{(i)}\longrightarrow Y^{(i+1)}$ for all $i$ so that for all $k$
\par\noindent
(1) $\HC_m(Y^{(k)};S')\leq (1-{1\over 2*12^m})^{k-1}\HC_m(Y;S')+\epsilon;$
\par\noindent
(2) Let $\Theta_k$ denote $\theta_k\circ\theta_{k-1}\circ\ldots \theta_1$. For each $y\in Y$ $dist(y,\Theta_k(y))\leq I_2(m)\HC_m(Y;S')+\epsilon$.
\end{lemma}

\begin{proof}
We repeatedly apply Lemma \ref{stepsY}. Inductively $\theta_i$ is the map $\theta$ from the remark
for $Y=Y^{(i)}$, and $Y^{(i+1)}=\tilde Y$. The property (1) immediately follows from the remark. Property (2) is proven exactly as (\ref{eq:cumulative distance}).
\end{proof}

\subsection{The case of $\HC_m(Y_{K+1})<\epsilon_0$}
\label{final step}
Let $X' = X \setminus \bigcup_{k=1}^K H_k$.
To finish the proof we need to define 
an extension $F: X' \longrightarrow S$ of map $f_{K+1}: Y_{K+1} \longrightarrow S$
with $\HC_{m}(Y_{K+1}) < \epsilon_0$,
where we can make $\epsilon_0$
arbitrarily small (in particular, much smaller than $\epsilon$) by picking 
$K$ sufficiently large.

By Kadec-Snobar estimate \cite[Theorem III.B.10]{Wo} there exists a 
projection map $\Pi_{S'}: S \longrightarrow S'$, such that $||\Pi(v)|| \leq \sqrt{N} ||v||$,
where $N$ is the dimension of $S'$. 
Let $\Pi_U: S' \longrightarrow U \cap S'$ denote the radial 
retraction onto the ball $U \cap S'$. We define map $\Pi=\Pi_U \circ \Pi_{S'}$.
We have that the Lipschitz constant of $\Pi$ satisfies 
$Lip(\Pi) \leq const(N)$.
Let $\{ B(p_i,r_i)\}$, $p_i \in S'$, be a finite collection of open balls
covering $Y_{K+1}$ with $\sum r_i^m \leq 2\epsilon_0$.
Let $V$ be an open set with $\bigcup_i B(p_i,r_i) \subset V  \subset \bigcup_i B(p_i,2r_i)$ 
and define a map $\Phi: \bigcup B(p_i,2r_i) \longrightarrow S$ 
setting $\Phi(x) = (1-\phi(x)) x + \phi(x) \Pi(x) $,
where $\phi: S \longrightarrow [0,1]$ is a continuous function,
such that $\phi(x) = 0$ for $x \in \bigcup B(p_i,r_i) $
and $\phi(x) = 1$ for $x \in \partial V$.

Observe that we have $\HC_{m}(\Phi(F_K(X)\cap V)) \leq const(N)\epsilon_0$.
Also, $\Phi(F_K(X)\cap V)$ is $const(N)R'$-close to $V$, and, therefore, to
$Y_K$, where $R'=const(N)\max_i r_i\leq const(N)\epsilon_0^{1\over m-1}$.
Finally, we apply Federer-Fleming deformation Lemma \ref{Federer-Fleming}
to define map $F$ that extends $\Phi \circ F_K$ restricted to $F_K^{-1}(\partial V \cap X)$
to a map on $X \setminus F_K^{-1}(V)$. By choosing $\epsilon_0(N)$ sufficiently small we can guarantee that 
$\HC_{m+1}(F(X \setminus V)) < \epsilon/2$
and $F(X \setminus V) \setminus  N_R(Y_K)$
is contained in an $(\lceil m \rceil-1)$-dimensional
CW complex $W$ for $R< \epsilon/2$.

This finishes the proof of Proposition \ref{reduction isop} and, therefore, Theorem \ref{isoperimetric}.

Similarly, we can project $Y_{K+1}$ to finite-dimensional $S'$ using $\Pi_{S'}$ and then use Lemma 2.5 and the remark after it to map $Y_{K+1}$ to a subset of the $(m-1)$-dimensional skeleton of the cubic complex $Q$ in $S'$. Note that $\Pi_{S'}$ maps each point $y\in Y^{K+1}$ to $\sqrt{N}r_i$-close point in $S'$, where $r_i\leq const (m)\epsilon_0^{1\over m}$ is the radius of the ball $B(x_i,r_i)$ that contains $y$. Composing the resulting map to the $(m-1)$-simplicial complex
with $\Theta_K$ from Lemma 3.6 we obtain a continuous map of $Y$ to a $(m-1)$-dimensional simplicial complex such that the distance between
each point $y$ and its image does not exceed $Const(m)\HC_m(Y)^{1\over m}$.

To prove Theorem \ref{main2} change the notation $X$ to $Y$ and observe that
one can assume that $Y$ is a subset of Banach space $L^{\infty}(Y)$, where
the inclusion sends each point $y\in Y$ to the distance function $d_y(x)$ on $Y$ defined as $d_y(x)=dist(x,y)$ (compare \cite{GromovFilling}).
Now the assertion at the end of the previous paragraph immediately
yields Theorem \ref{main2}.

Observe that our proof of Theorem \ref{main2} 
was almost a ``subset" of the proof of Theorem \ref{isoperimetric0}.
This is not surprising. Indeed, the general theory of absolute extensors for metric spaces implies that absolute extensors are precisely
absolute retracts, i.e. contractible ANR. In fact, if we are interested only in extensions of maps from some fixed  metric space $Y$, the requirement of contractibility
can be replaced by the weaker condition that the image of $Y$ is contractible in the considered space.
Correspondingly, our strategy of proving Theorem \ref{isoperimetric} was to embed $Y$ in an open set in $const(m)\HC_m(Y)^{1\over m}$-neighbourhood of $Y$, so that we have the desired
upper bound for $\HC_{m+1}$ of this set. Moreover, this set is contractible to a bounded subcomplex of $(m-1)$-skeleton of a subdivision of $n$-dimensional
linear space into cubes. (At this stage we already have Theorem \ref{main2}.) Adding the cone over this subset (or embedding it into the $m$-skeleton of the same cubic complex as in Lemma 2.4) we obtain an absolute extensor (or an absolute extensor for $Y$). We constructed explicit mappings from arbitrary $X$ to the constructed space, but could have just used the relevant results from Borsuk's theory of retracts.

\subsection{Proof of Proposition 1.3}

\vspace{0.1in}

Now we are going to prove Proposition 1.3. First, we are going to establish the isoperimetric inequality
stated in the proposition.

The proof is almost identical to that of Loomis-Whitney inequality
\cite{LW}.
Consider a covering of $\partial \Omega$ by cubes (nearly) 
realizing the Hausdorff content. 
We can approximate this covering arbitrarily well 
using cubes $\{ C_i \}$ in a sufficiently fine lattice,
with cubes of side length $2r$, where $r$ can be arbitrarily small. (Here we are using the assumption
that $n$ is the dimension of the cubes. For a cube with side length $R$, $R^n$ is its volume, and, therefore, $R^n$ 
will be only somewhat smaller than the sum of volumes of smaller cubes, providing that the union of smaller cubes 
includes the original cube, and is only ``slightly" larger than the ambient cube.)
Let $\pi_j$ denote the projection onto 
the $j$th coordinate $(n-1)$-plane and let $U_j = \pi_j(\bigcup C_i)$.
Note that $\Omega \subset \bigcap \pi_j^{-1} (U_j)$. 
Let $N$ denote the number of $r$-cubes in $\bigcap \pi_j^{-1} (U_j)$,
and $N_j$ be the number of $r$-cubes in 
$U_j$. By choosing the size $2r$ of the grid sufficiently 
small, we may assume that,
for an arbitrarily small $\epsilon>0$,
$\HC_{n-1}(\partial \Omega) \geq \HC_{n-1}(\pi_j(\partial \Omega)) 
\geq N_j r^{n-1}- \epsilon$.
By Theorem 2 in \cite{LW} $$N^{n-1} \leq N_1 N_2...N_n.$$
We conclude that $HC_n(\Omega) \leq N r^n 
\leq (\prod_{j=1}^n N_j r^{n(n-1)})^{1\over n-1} \leq (HC_{n-1}(\partial\Omega) + \epsilon)^{n\over n-1}$.

Now it remains to check what is going on in the case of coordinate cubes. If $\Omega$ is a coordinate cube with side length $R$,
then $\HC_n(\Omega)=({R\over 2})^n$ (Lemma 1.11). Also, $\HC_{n-1}(\partial\Omega)$ is not less than $\HC_{n-1}$ of
any of its $(n-1)$-dimensional faces, that are $(n-1)$-dimensional coordinate cubes with side length $R$. Therefore,
$\HC_{n-1}(\partial\Omega)\geq ({R\over 2})^{n-1}$. On the other hand, the closure of $\Omega$, including $\partial\Omega$, can be covered by just one closed metric ball (i.e. the coordinate cube) of radius ${R\over 2}$. Hence,  $\HC_{n-1}(\partial\Omega)^{n\over n-1}=({R\over 2})^n=\HC_n(\Omega)$.

\section{The proof of systolic inequalities.}

Here we prove the systolic inequalities stated in subsection 1.3. We will be using the upper bounds
for $UW_{m-1}$ provided by Theorems \ref{main2}, \ref{main}.

The proof is modelled on the argument from \cite{GromovFilling} used there to deduce the inequality $sys_1(M^n)\leq c(n)vol^{\frac{1}{n}}(M^n).$
First, observe that according to \cite{GromovFilling}, Appendix 1, Proposition (D) on p. 128, the Kuratowski embedding $f:X\longrightarrow L^\infty(X)$ is at the distance ${1\over 2}UW_{m-1}(X)$ from some
$(m-1)$-degenerate map $g:X\longrightarrow L^\infty(X)$, that is, a map $g$ which is a composition of a map $g_1$ of $X$ into a $(m-1)$-dimensional
polyhedron $K$, and a map $g_2$ of $K$ into $L^\infty(X)$. 
Let $W$ denote the mapping cylinder of $g_1$: the quotient space of the cylinder $X\times [0,1]$ by the quotient
map $g_1:X\times\{1\}\longrightarrow K$. Define $F:W\longrightarrow L^\infty(X)$ as $f$ on the ``bottom" $X\times\{0\}$ of $W$, $g$ (or, equivalently, $g_2$) on the ``top", and as straight line segments connecting $f(x)$ and $g(x)$ in $L^\infty(X)$ on all ``vertical" segments of $W$ ``above" $x\in X\times\{0\}$. 

Exactly as in the proof of Lemma 1.2.B from \cite{GromovFilling} one can prove that if $sys_1(X)> 3UW_{m-1}(X)$, then the classifying map
$Q:X\longrightarrow K(\pi_1(X),1)$ can be extended to $W$ by first mapping it as above to $L^\infty(X)$ and then extending the classifying map
defined on $X\times\{0\}\subset W\subset L^\infty(X)$. As in the proof of Lemma 1.2.B from \cite{GromovFilling} one considers a very fine triangulation
of $W$ and performs the extension to $0$-dimensional, then $1$-dimensional, then $2$-dimensional skeleta of the chosen triangulation of $W$. 
All new vertices of the triangulation are being mapped first to the nearest
points of $f(X)$, and then to $K(\pi_1(X),1)$ via the classifying
map $Q$. All $1$-dimensional simplices are first mapped to minimal geodesics
between the images of their endpoints in $f(X)$, and then to $K(\pi_1(X),1)$
using $Q$. Observe that the triangle inequality implies that their images in $f(X)$ have length $\leq UW_{m-1}(X)+\epsilon$, where $\epsilon$ can be made
arbitrarily small by choosing a sufficiently fine initial triangulation
of $W$. We observe that an easy compactness argument implies that there exists
a positive $\delta$ such that each closed curve of length $\leq 3UW_{m-1}(X)+\delta$
is still contractible. We choose $\epsilon$ above as ${\delta\over 3}$.
Now the boundary of each new $2$-simplex in $W$ has been already mapped to
a closed curve of length $\leq 3(UW_{m-1}(X)+\epsilon)$ in $X$ that is contractible
in $X$. So, we can map the corresponding $2$-simplex in $W$ by, first, contracting the image of
its boundary in $f(X)$ to a point, and then mapping the resulting $2$-disc in
$f(X)$ to $K(\pi_1(X),1)$
using the classifying map $Q$.

Finally, one
argues that the extension to the skeleta of all higher dimensions is always possible as $K(\pi_1(X),1)$ is aspherical, as the corresponding obstructions live in homology groups
of the pair $(W, f(X))$ with coefficients in trivial (higher) homotopy groups of the target space $K(\pi_1(X),1)$.

It remains to notice that this extension is impossible as the inclusion $X\times\{0\}\longrightarrow W$ is homotopic to $g=g_1\circ g_2$, and, therefore,
induces trivial homomorphisms of all homology groups in dimensions $\geq m$. Therefore, the existence of such an inclusion would contradict
the assumption that $X$ is $m$-essential.
\medskip

\medskip\noindent
{\bf Acknowledgements.} A part of this work was done during the authors' visit to the Research Institute for Mathematical Studies, Kyoto University, in July, 2018. Another part of this work was done while three of the authors (Yevgeny Liokumovuch, Alexander Nabutovsky and Regina Rotman) were members of the Institute for Advanced Study in January-April, 2019.
The authors would like to thank both the RIMS (Kyoto University) and the IAS for their kind hospitality. 

The authors would like to thank Gregory Chambers, Larry Guth, Stephane Sabourau, Christina Sormani and Robert Young for stimulating discussions.
 
 The research of Yevgeny Liokumovich was partially supported by NSF Grant DMS-1711053 and NSERC Discovery grant RGPAS-2019-00085.
 The research of Boris Lishak was supported by the Australian Research Council's Discovery funding scheme (project number DP160104502).
 The research of Alexander Nabutovsky was
 partially supported by his NSERC Discovery Grant RGPIN-2017-06068. The research of Regina Rotman was partially supported by her NSERC Discovery Grant RGPIN-2018-04523.

The authors are grateful to the anonymous referees for
numerous corrections and helpful suggestions.

\end{document}